\documentclass{article}
\usepackage{amsmath,amsthm,amssymb,amscd,fancybox,epsfig,float,subfigure}
\usepackage{graphics}
\usepackage{epstopdf}
\usepackage{multirow}
\usepackage{hhline}
\usepackage[nocompress]{cite}

\newcommand{\beqn}{\begin{equation}}
\newcommand{\eeqn}{\end{equation}}

\newcommand{\bR}{\mathbb{R}}

\newcommand{\cF}{\overline{\bf F}}
\newcommand{\cU}{\overline{\bf U}}

\newcommand{\bP}{{\bf P}}

\newcommand{\bq}{{\bf q}}

\newcommand{\bu}{{\bf u}}

\newcommand{\bphi}{{\boldsymbol \phi}}
\newcommand{\bmu}{{\boldsymbol \mu}}
\newcommand{\brho}{{\boldsymbol \rho}}
\newcommand{\bC}{{\bf C}}

\newcommand{\bF}{{\bf F}}
\newcommand{\bRes}{{\bf R}}
\newcommand{\bM}{{\bf M}}

\newcommand{\bx}{\mathbf{x}}
\newcommand{\bn}{\mathbf{n}}
\newcommand{\by}{\mathbf{y}}

\newtheorem{thm}{Theorem}

\newtheorem{lem}{Lemma}
\newtheorem{rem}{Remark}

\title{Integral equation methods for elastance and mobility problems
in two dimensions}

\author{Manas Rachh
\thanks{Courant Institute of Mathematical Sciences, 
         New York University. 
         New York, NY 10012-1110.
         ({\tt rachh@cims.nyu.edu}).}
\and L. Greengard
\thanks{Simons Foundation and Courant Institute of Mathematical Sciences, 
         New York University.
         ({\tt greengard@cims.nyu.edu}).}
}

\begin{document}

\maketitle

\begin{abstract}
We present new integral representations in two dimensions for the elastance problem
in electrostatics and the mobility problem in Stokes flow.
These representations lead to 
resonance-free Fredholm integral equations of the second kind and
well conditioned linear systems upon discretization. By coupling our integral 
equations with high order quadrature and fast multipole acceleration, large-scale
problems can be solved with only modest computing resources. 
We also discuss some applications of these boundary value problems in applied physics.
\end{abstract}

\section{Introduction}

A classical problem in electrostatics is the analysis of capacitance.
Briefly stated,
for an {\em open} system in two dimensions, this concerns a collection of 
$N$ disjoint, bounded regions, denoted by $D_j$, with boundaries 
$\Gamma_j$, all of which are assumed to be perfect conductors. 
Setting the potential (the voltage)
on the $j$th conductor to $\phi_j$ for $j = 1,\dots,N$, 
one would like to determine the 
net charge $q_j$ which accumulates on the conductors $D_j$ for $j = 1,\dots,N$.
Since electrostatics is governed by a linear partial differential
equation (the Laplace equation), there is a matrix, denoted by $\bC$, such that
\begin{equation}
\bq = \bC \, \bphi \, ,
\label{capeq}
\end{equation}
where $\bphi = (\phi_1,\dots,\phi_N)$
and $\bq = (q_1,\dots,q_N)$.
The matrix $\bC$ is referred to as the {\em capacitance matrix}.

Less well-studied is the converse problem, where a fixed amount of charge is 
placed on each of the $N$ conductors, and the goal is to determine the 
corresponding, unknown, potentials. 
The matrix corresponding to that mapping is called the 
{\em elastance matrix} \cite{Smythe1975}, denote by $\bP$, with 
\[ \bphi = \bP \, \bq \, .\]
We should note that the goal is achievable, since the voltage on a perfect 
conductor is constant from Maxwell's equations \cite{JACKSON}. 
$\bP$ is the inverse of $\bC$, in a suitably-defined space, discussed in greater 
detail in the next section. 
Given either the capacitance or elastance matrix, it is straightforward
to compute the electrostatic energy $E$ of the system
\cite{JACKSON}, since
\[ E =  \frac{1}{2} \bphi^T \bq = \frac{1}{2} \bphi^T \bC \bphi = 
\frac{1}{2} \bq^T \bP \bq \, . \]
In some contexts, particularly in chip design, the capacitance problem is 
more typical \cite{Kapur00,Chew2000,Tausch1998,Weeks1970}. In others,
including some quantum mechanical settings,
the elastance problem arises more naturally 
\cite{Chang2009,VanderWiel2003,Koch2007}. The elastance matrix
is sometimes referred to as the charging energy matrix.

A similar duality exists in problems of Stokes flow. Given
$N$ disjoint, rigid bodies, denoted by $D_i$, with boundaries $\Gamma_i$, with
prescribed translational and rotational velocities, $(\mathbf{v}_i,\omega_i)$, 
the {\em resistance problem} consists
of determing the corresponding forces and torques $(\bF_i,T_i)$ 
on each of the bodies. The
{\em mobility problem} is the reverse; given 
prescribed forces and torques on each of the rigid bodies, find the corresponding 
velocities (see, for example,
\cite{Delong2014,Ichiki2001,kim2005microhydrodynamics,pozrikidis1992boundary}).
The mappings $\bRes$ and $\bM$ such that
\[ \cF = \bRes \, \cU \quad{\rm and}\quad \cU = \bM \, \cF   \]
are known as the resistance and mobility tensors.
Here, $\cU = (\mathbf{v}_1,\omega_1,\dots,\mathbf{v}_N,\omega_N)$ and 
$\cF = (\bF_1,T_1,\dots,\bF_N,T_N)$.

Reformulating the governing partial differential equation as
a boundary integral equation is a natural approach for the above problems,
since this reduces the dimensionality of the problem (discretizing the 
boundaries $\Gamma_i$ alone) and permits high order accuracy to be achieved
in complicated geometries. Moreover, boundary
integral equations can be solved in optimal or nearly optimal time using 
suitable fast algorithms \cite{CHEW,LIU,acta1997,Nishimura}, and
satisfy the far field boundary conditions
necessary to model an open system without the need for artificial truncation
of the computational domain. 

In this paper, we will restrict our attention largely
to the formulation of suitable integral equations for the elastance
and mobility problems.
There is a substantial literature on
the development of well-conditioned second-kind Fredholm equations to address
the capacitance problem, which we do not seek to review here. 
We simply note that to apply $\bC$ to a vector $\bphi$ is equivalent
(in the two-dimensional setting) to solving the 
Dirichlet problem:
\begin{align}
\Delta u(\bx) & =  0\quad \bx\in 
\bR^2 \setminus \left( \cup_{i=1}^N \overline{D_i} \right)  \\
u|_{\Gamma_{j}} & =  \phi_{j} \, ,
\end{align}
together with the radiation condition that
$u(\bx)$ be bounded as $\left|\bx\right|\to\infty$.
The boundedness of $u(\bx)$ enforces charge neutrality on the collection of 
conductors. To see this, note that standard multipole estimates imply that
\[ u(\bx) \to  C + \frac{Q}{2 \pi} \log |\bx|\]
as $\left|\bx\right|\to\infty$, where $C$ is constant and 
$Q$ is the net charge induced on all the conductors. Thus, 
there would be logarithmic growth of the potential at infinity if 
the system were not charge neutral.

\begin{rem}
It is worth noting that the constant $C$ {\em cannot} be specified independently.
If, for example, $\phi_j$ were set to 1 on all conductors, the solution
to the Dirichlet problem must be $u(\bx) = 1$ (under the assumption of
charge neutrality), so that $C=1$.
\end{rem}

The Dirichlet problem, as noted above, is well-known to have a unique solution and
a variety of well-conditioned integral equations have been
derived for its solution (see
\cite{GGM,Helsing2005,mikhlin1964integral,Nabors1994,Tausch1998}
and the references therein).
The resistance problem involves solving the Stokes equations
with boundary conditions imposed on the velocity, for which there are,
again, a large number of well-conditioned formulations
\cite{Biros2003,GKM,kim2005microhydrodynamics,Power1987,pozrikidis1992boundary}.

Suitable integral representations have been developed for 
both the elastance and mobility
problems, often in the form of 
first kind integral equations
\cite{Cortez2005} or second kind integral equations with $N$ additional
unknowns and $N$ additional constraints
\cite{Kropinski1999}. While these have been shown to be very effective,
when $N$ is large and the geometry is complex, it is advantageous
to work with formulations that are both well-conditioned (formulated
as second kind boundary integral equations) and free of 
additional unknowns and constraints.
We develop such an approach for the electrostatic problem first, 
in section \ref{sec:elast},
followed by the Stokes mobility problem in section \ref{sec:mob}.
We illustrate their effectiveness with numerical
examples in section \ref{sec:num} and discuss generalizations 
in section \ref{sec:concl}. 

\begin{rem}
We note that the elastance 
problem can be interpreted as a special case
of a modified Dirichlet problem, and second kind Fredholm
integral equations for the modified Dirichlet problem are developed and discussed in 
\cite{mikhlin1964integral}.
In our representation, we compute the physical
charge density on each conductor directly and stably.
For the representation discussed in \cite{mikhlin1964integral}, the charge density
is given in terms of a hypersingular integral.

Second kind integral equations for mobility problems without additional constraints
were developed earlier by 
Kim and Karrila \cite{karrila1989integral}, using the Lorentz reciprocal identity. 
Our derivation, which leads to essentially the same integral equation, 
is direct - based on the physical principle that the 
interior of a rigid body must be stress free. 
Finally, the mobility problem can also be solved 
using a double layer representation without additional unknowns.
For a detailed discussion of this representation, 
we refer the reader to \cite{pozrikidis1992boundary,af2014fast}. 
Our formulation has the advantage that certain derivative quantities, 
such as fluid stresses, can be computed with weakly singular instead of hypersingular
kernels. 
\end{rem}

\section{The Elastance Problem} \label{sec:elast}

Given the set of conductors
$\left\{ D_{i}\right\} _{i=1}^{N}$, let us assume that the boundaries 
$\Gamma_{i}=\partial{D_{i}}$ are positively oriented.
We will denote by $\Gamma$ the total boundary $\Gamma=\cup_{i=1}^{N}\Gamma_{i}$ 
and by $\bn_{\bx}$ the
outward normal at $\bx\in\Gamma$. We let 
$E= \bR^2 \setminus \left(\cup_{i=1}^{N} \overline{D_{i}}\right)$ denote
the exterior domain. For the sake of simplicty, we assume that the 
conductors have smooth boundaries (Fig. \ref{Fig:1}.)

\begin{figure}[H]
\centering
\includegraphics[scale=0.3]{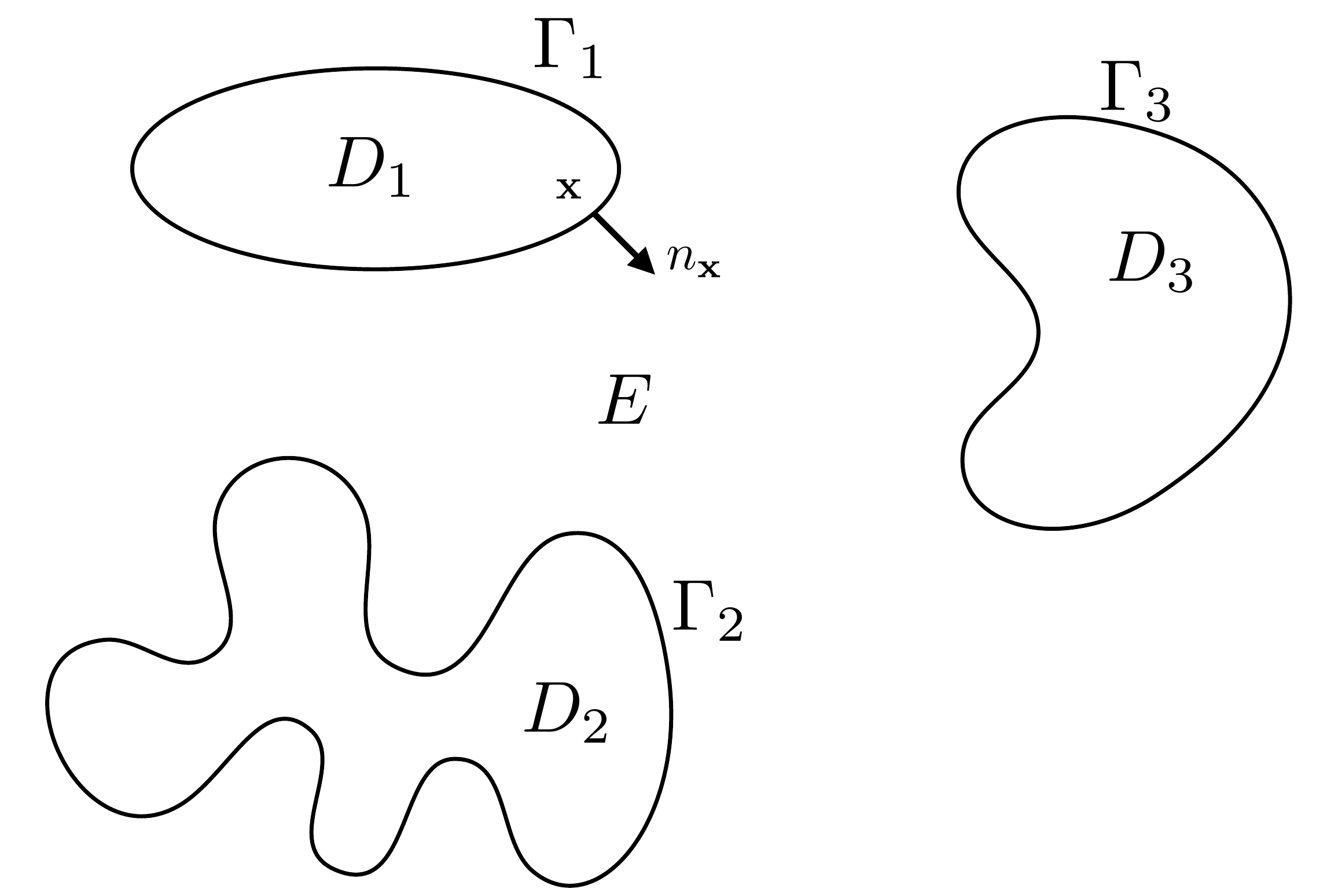}
\caption{Three, smooth bounded conductors in the plane, with the exterior
domain denoted by $E$.
For $\bx$ on the boundary, $\bn_{\bx}$ represents the outward normal.
\label{Fig:1}}
\end{figure}

Application of the elastance matrix to a vector of charge strengths
$\mathbf{q}=\left(q_{1},q_{2}\ldots q_{N}\right)$ is equivalent to
the solution of the following boundary value problem for the potential
$u(\bx)$ in the exterior domain $E$: 
\begin{align}
\Delta u(\bx) & =  0\qquad \bx\in E\label{eq:Harmonicity}\\
u|_{\Gamma_{j}} & =  \phi_{j}\label{eq:Equipotential}\\
-\int_{\Gamma_{j}}\frac{\partial u}{\partial n}ds_{\bx} & = q_{j}\label{eq:ChargeConstraints}\\
u\left(\mathbf{x}\right)&\to 0  \mbox{\quad as }
\left|\mathbf{x}\right|\to\infty \, . \label{eq:GrowthAtInfty}
\end{align}
Here, $u(\bx)$ and the constants $\left\{ \phi_{j}\right\} _{j=1}^{N}$
are unknown. As noted in the introduction, it is a consequence of the 
Maxwell equations that the potential on each
distinct
conducting surface is constant, so that the boundary 
condition \eqref{eq:Equipotential} corresponds to the physical problem
of interest. 
\eqref{eq:ChargeConstraints} enforces the desired charging of the 
individual conductors, and 
\eqref{eq:GrowthAtInfty} corresponds to setting the potential
at infinity to zero (ground). 

\begin{rem} ({\bf Charge neutrality}):
It is often said that the elastance matrix $\bP$ is the inverse of 
the capacitance matrix $\bC$. Unfortunately,
it is straightforward to verify that, for the vector of potential
values $\bphi_0 = (1,\dots,1)$, we have $\bC \bphi_0 = 0$, so that
$\bC$ is not actually invertible.
Likewise, the elastance
boundary value problem, as stated above, cannot be solved unless
$\bq=\left(q_{1},q_{2}\ldots q_{N}\right)$ satisfies 
\begin{equation}
\sum_{j=1}^{N} q_j = 0.
\label{eq:ChargeConservation}
\end{equation}
Otherwise, $u(\bx)$ would have logarithmic growth at infinity.
That is the sense in which $\bP$ is the inverse of $\bC$ - as a map 
defined on the space of mean zero vectors in $\bR^N$.
\end{rem}

To prove uniqueness for the elastance problem, we will need the following
lemma \cite{kress1999linear}.

\begin{lem} \label{lem:growth}
Let $u$ be a harmonic function in the exterior domain $E$ defined above,
satisfying the condition \eqref{eq:GrowthAtInfty}.
Let $B_{R}\left(0\right)$ be the ball of radius $R$ centered at
the origin and let $\partial B_{R}\left(0\right)$ be its boundary. Then,
there exist $M,R_{0}$ such that $\sup_{\partial B_{R}\left(0\right)}\left|\nabla u\right|\leq\frac{M}{R^{2}}$
for all $R\geq R_{0}$.
\end{lem}

\begin{lem}
(Uniqueness). Suppose that $u$ satisfies
equations \eqref{eq:Harmonicity}, \eqref{eq:Equipotential}, 
\eqref{eq:ChargeConstraints}
and \eqref{eq:GrowthAtInfty}, with $q_{i}=0$ for $i=1,\dots,N$. 
Then, $u(\bx)\equiv0$ in the exterior domain $E$.
\end{lem}

\begin{proof}
For sufficiently large $R$, we may write
\[
0=\int_{E\cap B_{R\left(0\right)}}u\Delta u \,dV =\int_{\partial B_{R}\left(0\right)}u 
\frac{\partial u}{\partial n}ds_{\bx} -\sum_{i=1}^{N}\int_{\Gamma_{i}}u
\frac{\partial u}{\partial n}ds_{\bx} -\int_{E\cap B_{R}\left(0\right)}\left|\nabla u\right|^{2}dV.
\]
Since $u(\bx)$ takes on some constant value $\phi_i$ on $\Gamma_i$, we may write
\begin{align*}
\int_{E\cap B_{R}\left(0\right)}\left|\nabla u\right|^{2}dV & =
\int_{\partial B_{R}\left(0\right)}u\frac{\partial u}{\partial n}ds_{\bx} -
\sum_{i=1}^{N}\phi_{i}\int_{\Gamma_{i}}\frac{\partial u}{\partial n}ds_{\bx} \quad \\
 & =\int_{\partial B_{R}\left(0\right)}u\frac{\partial u}{\partial n}ds_{\bx} \, ,
\end{align*}
since the $q_{i}$ are all zero.
From Lemma  \ref{lem:growth}, the boundedness of $u$, and the 
monotone convergence theorem, it is easy to see that
\[
\int_{E}\left|\nabla u\right|^{2}dV=
\lim_{R\to\infty}\int_{E\cap B_{R}\left(0\right)}\left|\nabla u\right|^{2}dV =
\lim_{R\to\infty}\int_{\partial B_{R}\left(0\right)}u\frac{\partial u}{\partial n}ds_{\bx} =0 .
\]
Thus, $\nabla u\equiv\mathbf{0}$ in $E$ and $u$ must be a constant. 
From the decay condition at infinity, $u\equiv0$ as desired.
\end{proof}

To develop an integral equation for the elastance problem, 
we will use the language of scattering theory.
That is, we will construct an ``incident field'' which satisfies the 
charging conditions \eqref{eq:ChargeConstraints} but not the boundary
conditions \eqref{eq:Equipotential}. We will then
solve for a ``scattered'' field which forces the 
conductors to be equipotential surfaces without changing the net charge
on any of the $\Gamma_i$. (This will also
yield a proof of existence of solutions.)

\begin{rem}
In physical terms, one can think of the problem as follows:
imagine that
we simply deposit charge uniformly on each conducting surface $\Gamma_i$
to satisfy the charging condition. 
This will be our incident field.
The charges will then redistribute themselves on
each $\Gamma_i$ so that they are equipotential surfaces. The total
field will be defined by that new, equilibrated charge distribution.
\end{rem}

\subsection{Mathematical preliminaries}

Let $\gamma$ be a smooth closed curve in $\bR^{2}$ and let $D^{\mp}$
denote the domains corresponding to the interior and exterior of 
$\gamma$. Let $\bn_{\bx}$ be the unit outward normal
to the curve $\gamma$ and let $\bn_{0}=\bn_{\bx_0}$ for
$\bx_{0}\in\gamma$ . Let $\mu:\gamma\to\bR$ be a continuous
function. The single layer potential is defined by
\begin{equation}
S_{\gamma}\mu\left(\mathbf{x}\right)=\int_{\gamma}
G\left(\mathbf{x},\mathbf{y}\right)\mu\left(\mathbf{y}\right)ds_{\by} \, ,
\end{equation}
where $G(\bx,\by)$ is the 
fundamental solution for the Laplace equation in free space:
\begin{equation}
G\left(\mathbf{x},\mathbf{y}\right)=-\frac{1}{2\pi}\log\left|\mathbf{x}-\mathbf{y}\right| \, .
\end{equation}

\begin{lem} 
{\rm \cite{guenther1988partial,kress1999linear,mikhlin1964integral}} 
Let $S_{\gamma}\mu\left(\mathbf{x}\right)$ be a single layer potential
with charge density $\mu$ defined on $\gamma$. Then
$S_{\gamma}\mu\left(\mathbf{x}\right)$ is harmonic
in $\bR^{2}\backslash\gamma$ and continuous
in $\bR^{2}$. The single layer potential satisfies
the jump relations
\begin{equation}
\lim_{\substack{\mathbf{x}\to\mathbf{x}_{0}\\ \mathbf{x}\in D^{\pm}}} 
\frac{\partial S_{\gamma}\mu\left(\mathbf{x}\right)}{\partial n_{0,\pm}}=
\mp\frac{1}{2}\mu\left(\mathbf{x}_{0}\right)+
\oint_{\gamma} \frac{\partial G\left(\mathbf{x}_{0},\mathbf{y}\right)}{\partial n_{0}}
\mu\left(\mathbf{y}\right)ds_{\by} \label{SLnormaljump}
\end{equation}
where $\oint_{\gamma}$ indicates the principal value integral
over the curve $\gamma$ and
the subscripts $-$ and $+$ denote the limits of the integral from
the interior and exterior side, respectively. Furthermore,
\begin{equation}
-\int_{\gamma}\frac{\partial S_{\gamma}\mu\left(\mathbf{x}\right)}{\partial n_{\bx,+}}ds_{\bx}
 =\int_{\gamma}\mu\left(\bx\right)ds_{\bx}\label{eq:SLGaussLaw1}\, ,\quad
\int_{\gamma}\frac{\partial S_{\gamma}\mu\left(\mathbf{x}\right)}{\partial n_{\bx,-}}ds_{\bx} =0.
\end{equation}
For a closed curve 
$\omega\subset D^{+}$, we also have
\begin{equation}
\int_{\omega}\frac{\partial S_{\gamma}\mu\left(\mathbf{x}\right)}{\partial n_{\bx}}ds_{\bx}
=0 . \label{eq:SLGaussLaw2}
\end{equation}
Finally,
\begin{equation}
\left|S_{\gamma}\mu\left(\mathbf{x}\right)+\frac{1}{2\pi} Q
\log{\left(\mathbf{x}\right)}\right|\to0\quad\mbox{as }\left|\mathbf{x}
\right|\to\infty , \label{eq:SLGrowthAtInfty}
\end{equation}
where 
\[ Q = \int_{\gamma} \mu\left(\mathbf{y}\right)ds_{\by} . \]
\end{lem}

The double layer potential $D_{\gamma}\mu(\bx)$ 
is the potential due to a surface density of dipole sources on $\gamma$,
aligned in the normal direction to the curve:
\begin{equation}
D_{\gamma}\mu\left(\mathbf{x}\right)=\int_{\gamma}\frac{\partial G\left(\mathbf{x},\mathbf{y}\right)}{\partial n_{\by}}\mu\left(\mathbf{y}\right)ds_{\by}
\end{equation}
\begin{lem}
\cite{guenther1988partial,kress1999linear,mikhlin1964integral} 
Let $D_{\gamma}\mu\left(\mathbf{x}\right)$ be a double layer potential.
Then,
$D_{\gamma}\mu\left(\mathbf{x}\right)$ is harmonic in $\bR^{2}\backslash\gamma$
and satisfies the jump relations:
\begin{equation}
\lim_{\substack{\mathbf{x}\to\mathbf{x}_{0}\\
\mathbf{x}\in D^{\pm}}} D_{\gamma}\mu=\pm\frac{1}{2}\mu\left(\mathbf{x}_{0}\right)+
\oint_{\gamma}\frac{\partial G\left(\mathbf{x}_{0},\mathbf{y}\right)}{\partial n_{\by}}
\mu\left(\mathbf{y}\right)ds_{\by}\label{eq:DLP0} \, .
\end{equation}
Furthermore,
\begin{align}
\int_{\gamma}\frac{\partial G\left(\mathbf{x},\mathbf{y}\right)}{\partial n_{\by}}ds_{\by} &=
\begin{cases}
-1 & \quad\mathbf{x}\in D^{-}\\
0 & \quad\mathbf{x}\in D^{+}
\end{cases}\label{eq:DLP1}\\
\oint_{\gamma}\frac{\partial G\left(\mathbf{x},\mathbf{y}\right)}{\partial n_{\by}}
ds_{\by} &=-\frac{1}{2}\quad\mathbf{x}\in\gamma \, , \label{eq:DLP2}
\end{align}
and
\begin{equation}
\left|D_{\gamma}\mu\left(\mathbf{x}\right)\right|\to0\quad\mbox{as }\left|\mathbf{x}\right|\to\infty \, . \label{eq:DLGrowthAtInfty}
\end{equation}
\end{lem}

\subsection{Charging the boundaries with an incident field \label{sec:incelast}}

In the elastance problem, perhaps the 
simplest way to allocate the net charge $q_i$ to the boundary $\Gamma_i$ is 
to define a constant charge density 
\begin{equation} 
\sigma_i(\bx) = \frac{q_i}{|\Gamma_i|}  \, ,
\label{eq:sigma}
\end{equation}
for $\bx$ on the curve $\Gamma_i$, where $|\Gamma_i|$ denotes its length.
We can then define
$\sigma\left(\mathbf{x}\right)=\left(\sigma_{1}\left(\mathbf{x}\right),\sigma_{2}\left(\mathbf{x}\right)\ldots\sigma_{N}\left(\mathbf{x}\right)\right)$
and
\begin{align}
u_{inc}\left(\mathbf{x}\right) & ={S}_{\Gamma}\mathbf{\sigma}\left(\mathbf{x}\right)\label{eq:IncRepresentation}
\end{align}
where ${S}_{\Gamma}$
is the operator given by
\begin{align*}
{S}_{\Gamma}\sigma\left(\mathbf{x}\right) & =\sum_{j=1}^{N}S_{\Gamma_{j}}\sigma_{j}\left(\mathbf{x}\right)=\sum_{j=1}^{N}\int_{\Gamma_{j}}G\left(\mathbf{x},\mathbf{y}\right)\sigma_{j}\left(\mathbf{y}\right)ds_{\by} \, .
\end{align*}
From \eqref{eq:SLGaussLaw1} and \eqref{eq:SLGaussLaw2}, we have
\begin{equation}
-\int_{\Gamma_{j}}\frac{\partial u_{inc}\left(\mathbf{x}\right)}{\partial n}ds_{\bx} 
 =\int_{\Gamma_{j}}\sigma_{j}\left(\mathbf{x}\right)ds_{\bx}
= \frac{q_{j}}{|\Gamma_j|} 
 \int_{\Gamma_{j}} \, ds_{\bx} = q_j\, ,
\label{eq:ChargeImposition}
\end{equation}
for $j=1,2,\dots,N$.
Thus, $u_{inc}$ 
satisfies the charge constraints \eqref{eq:ChargeConstraints}.

\subsection{The scattered field}

We now seek 
a scattered field 
\begin{align}
u_{sc}\left(\mathbf{x}\right) & ={S}_{\Gamma}\mu\left(\mathbf{x}\right)
\label{eq:ScRepresenation}
\end{align}
such that
$u(\bx) = u_{inc}(\bx) + u_{sc}(\bx)$,
where $\mu(\bx) = (\mu_1(\bx), \mu_2(\bx),\dots,\mu_N(\bx))$
with $\mu_j$ an unknown charge density on
the boundaries $\Gamma_{j}$.
To ensure that no additional net charge has been introduced on any
of the conductors, we impose the $N$ integral constraints on
$\mu\left(\mathbf{x}\right)$:
\begin{align*}
\int_{\Gamma_{j}}\mu_{j}\left(\mathbf{x}\right)ds_{\bx} & =0 \, .
\end{align*}
If we can find such functions $\mu_j(\bx)$, then
\begin{align}
u(\bx) & =u_{inc}(\bx)+u_{sc}(\bx)=
S_{\Gamma}(\mu+\sigma)(\bx) \, , \label{eq:TotRepresentation}
\end{align}
solves the elastance problem.
Physically, $(\mu_j + \sigma_j)(\bx)$ is the final
charge density on $\Gamma_j$ once the total charge place on the
boundary has equilibrated to 
enforce the perfect conductor boundary condition \eqref{eq:Equipotential}.

\subsection{Formulation as a Neumann problem}

Letting $u(\bx) = u_{inc}(\bx) + u_{sc}(\bx)$, note first that
the corresponding potential is also defined inside each 
conductor. Rather than imposing the boundary condition 
\eqref{eq:Equipotential} from the exterior, however, 
we can make use of the fact that the
electric field inside each conductor given by $\nabla u\left(\mathbf{x}\right)$
must be identically zero. Thus, we may impose 
the interior Neumann boundary conditions
\begin{equation}
\frac{\partial u}{\partial n_{-}}\left(\mathbf{x}\right)\equiv0
\end{equation}
for $\mathbf{x}\in\Gamma$. 
Using \eqref{SLnormaljump},
we obtain the following second kind integral equation:
\begin{equation}
\left(\frac{1}{2}I+K\right)\mu\left(\mathbf{x}\right) = - \left(\frac{1}{2}I+K\right)\sigma\left(\mathbf{x}\right)
\label{eq:IntegralEq}
\end{equation}
for $\bx\in\Gamma$, subject to the constraints
\begin{equation}
\int_{\Gamma_{j}}\mu_{j}\left(\mathbf{x}\right) ds_{\bx}=0,
\label{eq:IntEqConstraints}
\end{equation}
for $j=1,2,\ldots N$.
Here,
\begin{align}
I & =\left(\begin{array}{cccc}
I_{1}\\
 & I_{2}\\
 &  & \ddots\\
 &  &  & I_{N}
\end{array}\right) \, , \\
K & =\left(\begin{array}{cccc}
K_{1,1} & K_{1,2} & \ldots & K_{1,N}\\
K_{2,1} & K_{2,2} & \ldots & K_{2,N}\\
\vdots & \vdots & \ddots & \vdots\\
K_{N,1} & K_{N,2} &  & K_{N,N}
\end{array}\right) \, ,
\end{align}
$I_{i}: C^{0,\alpha}\left(\Gamma_{i}\right)\to C^{0,\alpha}\left(\Gamma_{i}\right)$ is the identity map, 
$K_{i,j}:C^{0,\alpha}\left(\Gamma_{j}\right)\to C^{0,\alpha}\left(\Gamma_{i} \right)$, and 
$K_{i,i}:C^{0,\alpha}\left(\Gamma_{i}\right)\to C^{0,\alpha}\left(\Gamma_{i} \right)$
are the operators given by

\begin{align}
K_{i,j}\sigma & =\int_{\Gamma_{j}}\frac{\partial G\left(\mathbf{x},\mathbf{y}\right)}{\partial n_{\bx}}\sigma\left(\mathbf{y}\right)ds_{\by} 
\quad \mathbf{x} \in \Gamma_{i}\\
K_{i,i}\sigma & =\oint_{\Gamma_{i}}\frac{\partial G\left(\mathbf{x},\mathbf{y}\right)}{\partial n_{\bx}}\sigma\left(\mathbf{y}\right)ds_{\by} 
\quad \mathbf{x} \in \Gamma_{i} \, ,
\end{align}
where $C^{0,\alpha} \left(\Gamma \right)$ is the H\"{o}lder space
with exponent $\alpha$ and $\alpha > 0$. For a related treatment of the capacitance problem, see
\cite{Tausch1998}.

\begin{thm}
Let $u\left(\mathbf{x}\right)$ be defined as in \eqref{eq:TotRepresentation}.
If $\mu\left(\mathbf{x}\right)$ solves equations \eqref{eq:IntegralEq}
and \eqref{eq:IntEqConstraints}, then $u\left(\mathbf{x}\right)$ solves
the elastance problem.
\end{thm}
\begin{proof}
Note first that 
$u\left(\mathbf{x}\right)$ is harmonic in $E$ by construction.
Using \eqref{eq:SLGaussLaw1} and \eqref{eq:SLGaussLaw2},
the choice of $\sigma$ in equation \eqref{eq:sigma} and the constraints
\eqref{eq:IntEqConstraints}, we see that $u\left(\mathbf{x}\right)$
satisfies the charge constraints \eqref{eq:ChargeConstraints}.
Furthermore, from equations \eqref{eq:ChargeConservation}
and \eqref{eq:SLGrowthAtInfty}, it follows that $u\left(\mathbf{x}\right)\to0$
as $\left|\mathbf{x}\right|\to\infty$. Since $u\left(\mathbf{x}\right)$ is harmonic
in $D_{i}$ and satisfies $\frac{\partial u}{\partial n_{-}}\equiv0$
on $\Gamma_{i}$, $u\left(\mathbf{x}\right)\equiv c_{i}$ for some constant
$c_i$ inside $D_{i}$.
By the continuity of the single layer potential, $u = c_{i}$
from the exterior side of $\Gamma_{i}$ as well.
\end{proof}

\begin{rem} ({\bf The adjoint operator}):
The operator $K$ in equation \eqref{eq:IntegralEq} is a compact operator
for smooth $\Gamma$. Hence, $\frac{1}{2}I+K$ is a Fredholm integral
equation of the second kind. To
study existence of solutions to $\left(\frac{1}{2}I+K\right)\mu=f$,
we shall study existence of solutions for the adjoint problem
$\left(\frac{1}{2}I+K^{*}\right)\mu=f$ instead, where
\begin{align}
K^{*} & =\left(\begin{array}{cccc}
K_{1,1}^{*} & K_{1,2}^{*} & \ldots & K_{1,N}^{*}\\
K_{2,1}^{*} & K_{2,2}^{*} & \ldots & K_{2,N}^{*}\\
\vdots & \vdots & \ddots & \vdots\\
K_{N,1}^{*} & K_{N,2}^{*} &  & K_{N,N}^{*}
\end{array}\right) \, . \label{eq:LAdjoint}
\end{align}
Here,
$K_{i,j}^{*}:C^{0,\alpha}\left(\Gamma_{j}\right)\to C^{0,\alpha}\left(\Gamma_{i}\right)$ and
$K_{i,i}^{*}:C^{0,\alpha}\left(\Gamma_{i}\right)\to C^{0,\alpha}\left(\Gamma_{i}\right)$
are the operators given by
\begin{align}
K_{i,j}^{*}\sigma & =\int_{\Gamma_{j}}\frac{\partial G\left(\mathbf{x},\mathbf{y}\right)}{\partial n_{\by}}\sigma\left(\mathbf{y}\right)ds_{\by}\label{eq:Kadj0}\\
K_{i,i}^{*}\sigma & =\oint_{\Gamma_{i}}\frac{\partial G\left(\mathbf{x},\mathbf{y}\right)}{\partial n_{\by}}\sigma\left(\mathbf{y}\right)ds_{\by}\label{eq:Kadj1}
\end{align}
It is straightforward to verify that
$K_{i,j}^{*}\sigma(\bx) =D_{\Gamma_{j}}\sigma(\bx)$ for
$\bx \in \Gamma_{i}$.

Let $\sigma=\left\{ \sigma_{i}\left(\mathbf{x}\right)\right\} _{i=1}^{N}$
where each $\sigma_{i}\left(\mathbf{x}\right)$, supported on $\Gamma_{i}$,
is constant. Then, using  \eqref{eq:DLP1},\eqref{eq:DLP2} and \eqref{eq:LAdjoint}, we may conclude that 
$\left(\frac{1}{2}I+K^{*}\right)\sigma=0$.
Thus, the dimension of the null space of $\frac{1}{2}I+K^{*}$ is at
least $N$. In fact, it is well-known that the dimension of the null space 
is exactly $N$ \cite{mikhlin1964integral, kress1999linear}. 
\end{rem}

\begin{rem}
$\frac{1}{2}I+K^{*}$ is the integral operator one would obtain in
seeking to impose
Dirichlet boundary conditions with the potential represented as
a double layer potential. The double layer potential operator for
the exterior, however, is range deficient. It cannot represent 
a harmonic function
$u\left(\mathbf{x}\right)$ in the exterior which is generated 
by net charge 
in any of the domains $D_i$. 
To see this, note that the net charge is 
$-\int_{\Gamma_{i}}\frac{\partial u}{\partial n}$ from
\eqref{eq:SLGaussLaw1}, but that 
the double layer potentials satisfies 
$\int_{\Gamma_{i}}\frac{\partial u}{\partial n}=0$
for $i=1,2,\ldots N$ from \eqref{eq:DLP1}.
\end{rem}

\subsection{Existence of solutions \label{sec:extelast}}
From the preceding discussion (the existence of a nontrivial nulllspace),
it follows from the Fredholm alternative that 
$\left(\frac{1}{2}I+K\right)\mu=f$
has an $N$ dimensional space of solutions, so long as $f$ is in the
range of the operator $\frac{1}{2}I+K$. Using our representation for
the elastance problem, the right hand side in equation
\eqref{eq:IntegralEq} is certainly in the range of the operator $\frac{1}{2}I+K$.
The role of the additional $N$ integral constraints is, therefore, to 
pick out the unique one which doesn't alter the net charge on the $N$ conductors.
However, we do not wish to solve an overdetermined (non-square)
linear system.
If we simply discretize
the integral equation using, say a Nystr\"{o}m method, with $M$ points
on $\Gamma$, we would have to solve an $(M+N)\times M$ linear system to 
obtain the desired solution.
Instead, we propose to solve the integral equation
\[
\frac{1}{2}\mu_{i}(\bx)+\sum_{j=1}^{N}
K_{i,j}\mu_{j}(\bx)+\int_{\Gamma_{i}}\mu_{i}(\bx) ds_{\bx} 
= - \frac{1}{2}\sigma_{i}\left(\mathbf{x}\right)-\sum_{j=1}^{N}
K_{i,j}\sigma_{j}(\bx)
\]
for $\bx\in\Gamma_{i}$, or
\begin{equation}
\left(\frac{1}{2}I+K+L\right)\mu  = 
-\left(\frac{1}{2}I+K\right)\sigma\label{eq:ModRep2}
\end{equation}
where
\begin{align}
L & =\left(\begin{array}{cccc}
L_{1}\\
 & L_{2}\\
 &  & \ddots\\
 &  &  & L_{N}
\end{array}\right) \, ,
\end{align}
with $L_{i}:C^{0,\alpha}\left(\Gamma_{i}\right)\to C^{0,\alpha}\left(\Gamma_{i}\right)$ defined
by $L_{i}\mu_{i}(\mathbf{x})=\int_{\Gamma_{i}}\mu_{i}(\mathbf{y})ds_{\by}$. 

The following lemma shows that solving 
\eqref{eq:ModRep2}
is equivalent to solving \eqref{eq:IntegralEq} 
with constraints \eqref{eq:IntEqConstraints}. 
\begin{lem}
If $\mu$ solves equation \eqref{eq:ModRep2}, then $\mu$ solves equations
\eqref{eq:IntegralEq} and \eqref{eq:IntEqConstraints} \label{Lem:equivalence}
\end{lem}
\begin{proof}
Using equation \eqref{eq:SLGaussLaw2}, we observe that $\int_{\Gamma_{i}}K_{i,j}\mu_{j}\left(\mathbf{x}\right)ds_{\bx}=0$
for $j\neq i$. Furthermore, switching the order of integration in
$\oint_{\Gamma_{i}}K_{i,i}\mu_{i}\left(\mathbf{x}\right)ds_{\bx}$ and
using property \eqref{eq:DLP2} of the double layer potential, we see
that

\begin{align}
\oint_{\Gamma_{i}}K_{i,i}\mu_{i}\left(\mathbf{x}\right)ds_{\bx} & =\int_{\Gamma_{i}}\oint_{\Gamma_{i}}\frac{\partial G \left(\mathbf{x},\mathbf{y}\right)}{\partial n_{\bx}}\mu_{i}\left(\mathbf{y}\right)ds_{\by}ds_{\bx}\\
 & =\int_{\Gamma_{i}}\mu_{i}\left(\mathbf{y}\right)\oint_{\Gamma_{i}}\frac{\partial G\left(\mathbf{x},\mathbf{y}\right)}{\partial n_{\bx}}ds_{\bx}ds_{\by}\\
 & =-\frac{1}{2}\int_{\Gamma_{i}}\mu_{i}\left(\mathbf{y}\right)ds_{\by} \, .
\end{align}
Integrating expression \eqref{eq:ModRep2} on $\Gamma_{i}$, we may
conclude that
\begin{equation}
\left|\Gamma_{i} \right| \int_{\Gamma_{i}}\mu_{i}\left(\mathbf{x}\right)ds_{\bx}=0 \, .
\end{equation}
Thus, $L_{i}\mu_{i}(\mathbf{x})=0$, which implies that
$\mu$ satisfies the integral constraints \eqref{eq:IntEqConstraints}
and that 
\begin{equation}
\left(\frac{1}{2}I+K+L\right)\mu=\left(\frac{1}{2}I+K\right)\mu=-\left(\frac{1}{2}I+K\right)\sigma \, .
\end{equation}
\end{proof}
\begin{rem}
For further discussion of the solution of consistent linear systems with constraints
in the finite dimensional case, see \cite{NullSpacePaper}.
\end{rem}

The following lemma shows that the operator $\frac{1}{2}I+K+L$ has no null space.
\begin{lem}
 The operator $\frac{1}{2}I+K+L$ is injective.
\end{lem}
\begin{proof}
 Let $\mu \in \mathcal{N}\left(\frac{1}{2}I + K +L \right)$, i.e. it solves
 $\left(\frac{1}{2}I+K+L\right)\mu = 0$. Following the proof of Lemma \ref{Lem:equivalence},
 we conclude that $L \mu = 0$ and therefore $\left(\frac{1}{2}I + K\right)\mu = 0$. Let $u=S_{\Gamma} \mu$.  
 From the properties of the single layer potential 
 \begin{equation}
 \frac{\partial u}{\partial n}_{-} = \left(\frac{1}{2}I+K \right)\mu = 0 \, .
 \end{equation}
 By uniqueness of solutions to interior Neumann problem, we conclude that $u$ is a constant on each boundary
 component. Thus, $u$ solves the Elastance problem with $q_{i} = 0$, as $L \mu = 0$.  By uniqueness of
 solutions to the Elastance problem, we conclude that $u\equiv 0$ in $E$. Hence, 
 $\frac{\partial u}{\partial n}_{+} =0$. From the properties of the single layer,
 \begin{equation}
  \mu = \frac{\partial u}{\partial n}_{-} - \frac{\partial u}{\partial n}_{+} = 0 \, .
 \end{equation}
 Therefore, $\mathcal{N} \left(\frac{1}{2}I + K + L\right) = \left \{ 0 \right \}$.
\end{proof}
By the Fredholm alternative, we conclude that \eqref{eq:ModRep2} has a unique
solution $\mu$. 
\section{The mobility problem} \label{sec:mob}
Supose now that we have $N$ rigid bodies immersed in an incompressible Stokesian
fluid in $\bR^{2}$. Let $\bF_{i},T_{i}$ denote the force
and torque exerted on rigid body $D_i$ in a fluid which is otherwise assumed
to be at rest and let $\mathbf{v}_{i},\omega_{i}$ be the corresponding rigid body
motion, where $\omega_{i}$ is the angular velocity about the centroid
of $D_i$. The mobility matrix $\mathbf{M}\in\bR^{3N\times3N}$
is the linear mapping from the forces and torques on the rigid bodies
to the respective rigid body motions:
\[ \cU = \bM \cF  \]
where $\cU = (\mathbf{v}_1,\omega_1,\dots,\mathbf{v}_N,\omega_N)$ and 
$\cF = (\bF_1,T_1,\dots,\bF_N,T_N)$.

Referring to Fig. \eqref{Fig:1}, let $D_{i}$ now represent the rigid bodies
and let $E$ represent the Stokesian fluid with viscosity $\mu=1$.
Further, let us assume that there are no other volume forces on the fluid.
Let $\mathbf{u}(\mathbf{x})=(u_{1}(\mathbf{x}),u_{2}(\mathbf{x}))$
represent the fluid velocity in $E$ and let 
$(\bF_{1},T_{1},\ldots,\bF_{N},T_{N})$
be the force and torque exerted on the rigid bodies.
Let $\mathbf{x}^c_{i}=\frac{1}{\left|\Gamma_i \right|}\int_{\Gamma_{i}} \mathbf{x} ds_{\bx}$ 
be the centroid of $\Gamma_{i}$. Let $p$ be the fluid
pressure and let $\boldsymbol{\sigma}$ be the stress tensor associated with
the flow:
\begin{equation}
\sigma_{ij}=-p\delta_{ij}+\left(\frac{\partial u_{i}}{\partial x_{j}}+\frac{\partial u_{j}}{\partial x_{i}}\right)=-p\delta_{ij}+2e\left(\mathbf{u}\right)\label{StressTensor}
\end{equation}
where $\delta_{ij}$ is the Kronecker delta,
\begin{equation}
e\left(\mathbf{u}\right)=\frac{1}{2}\left(D\mathbf{u}+D\mathbf{u}^{T}\right)
\end{equation}
is the strain tensor associated with the flow, 
and $D\mathbf{u}$ is the gradient of $\mathbf{u}$.

On the surface of rigid bodies $\Gamma_{i}$, 
\begin{equation}
\mathbf{f}=\boldsymbol{\sigma}\cdot\mathbf{n}=\left[\begin{array}{cc}
\sigma_{11} & \sigma_{12}\\
\sigma_{21} & \sigma_{22}
\end{array}\right]\left[\begin{array}{c}
n_{1}\\
n_{2}
\end{array}\right]
\end{equation}
represents the surface force or surface traction exerted by the fluid
on $D_{i}$, where $\mathbf{n}$ is the outward normal to $\Gamma_{i}$. For notational convenience, let 
$\mathbf{x}^\perp = \left[\begin{array}{c}
-x_{2}\\x_{1} \end{array}\right]$ and $\nabla^{\perp} = \left[\begin{array}{c}
-\frac{\partial }{\partial x_{2}}\\ \frac{\partial }{\partial x_{1}} 
\end{array}\right] \, .$
Then $\mathbf{u}\left(\mathbf{x}\right)$ solves 
\cite{kim2005microhydrodynamics,pozrikidis1992boundary}
\begin{align}
-\Delta\mathbf{u}+\nabla p & =0\quad\mbox{in E}\label{eq:StokesFlowEq}\\
\nabla\cdot\mathbf{u} & =0\quad\mbox{in E}\label{eq:MassConservation}\\
\mathbf{u}\left(\mathbf{x}\right)|_{\Gamma_{i}} & =\mathbf{v}_{i}+\omega_{i}\left(\mathbf{x} - \mathbf{x}^{c}_i\right)^{\perp}\label{eq:RigidBodyMotion}\\
\int_{\Gamma_{i}}\mathbf{f}\,ds_{\bx}=\int_{\Gamma_{i}}\boldsymbol{\sigma}.\mathbf{n}\,ds_{\bx} & =-\mathbf{F}_{i}\label{eq:ForceCondition}\\
\int_{\Gamma_{i}}  
( \mathbf{f},
(\mathbf{x} -\mathbf{x}^{c}_i)^{\perp}) \, ds_{\bx} 
& =-T_{i}\label{eq:TorqueConditions}\\
\mathbf{u}\left(\mathbf{x}\right) & \to\mathbf{0}\quad
\mbox{as }\left|\mathbf{x}\right|\to\infty \, , \label{eq:NoFlowAtInfty}
\end{align}
where $\left(\mathbf{a}, \mathbf{b} \right)$ represents the Euclidean inner product for
vectors $\mathbf{a},\mathbf{b}\in \bR^{2}$.
Equations \eqref{eq:StokesFlowEq} and \eqref{eq:MassConservation} are the
governing equations for Stokes flow in domain $E$. 
Equation \eqref{eq:RigidBodyMotion}
enforces a rigid body motion on $D_i$, where 
$\mathbf{v}_{i},\omega_{i}$ are unknown.
Equations \eqref{eq:ForceCondition} and \eqref{eq:TorqueConditions}
state the net applied forces and torques are given by the {\em known} quantities
$(\bF_{i},T_{i})$. Finally,
\eqref{eq:NoFlowAtInfty} states that the fluid is at rest in the absense
of forcing. 
As a consequence of Stokes paradox,
there might not exist a solution to the set of equations described
above. In fact, it can be shown that 
\begin{equation}
\mathbf{u}\left(\mathbf{x}\right)=O\left(-\sum_{i=1}^{N}\mathbf{F}_{i}
\log{\left|\mathbf{x}\right|}\right) \, .
\end{equation}
Thus, a necessary condition for a solution to exist is that
\begin{equation}
\sum_{i=1}^{N} \mathbf{F}_{i} = 0 \, . \label{eq:NoNetForce}
\end{equation}
From 
\cite{kim2005microhydrodynamics,pozrikidis1992boundary}, it turns out that
\eqref{eq:NoNetForce} is also sufficient
for a solution satisyfing equation \eqref{eq:NoFlowAtInfty}.
To prove uniqueness for the mobility problem, we need the following
lemmas which can be found in \cite{UniquenessNollFinn}.

\begin{lem}
\label{lem:HarmonicProperty1}
If $h$ is a bounded harmonic function in $E$
and $n$ is an integer greater than $0$, 
with $h = O\left( r^{-n}\right)\;\;\text{as}\;\; r\to\infty$, then
\begin{equation} 
h(r,\theta) = \sum_{k=n}^{\infty} r^{-k} a_k \left(\theta \right) \, ,
\end{equation}
which converges uniformly outside $B_{R} \left(0\right)$ for some $R$.
\end{lem}

Let $\omega$ be the vorticity corresponding to the flow, defined by
\begin{equation}
\omega = \left( \nabla^{\perp} , \mathbf{u} \right) \, .
\end{equation}

\begin{lem}
If $\mathbf{u}$ satisfies equation \eqref{eq:NoFlowAtInfty}, then $\omega = O\left(r^{-1} \right)$ as $r \to \infty$.
\end{lem}

\begin{lem}
\label{lem:pressureVorticityBound}
If $\omega = O\left(r^{-n}\right)$ as $r\to\infty$ for integer $n>0$, 
then $p = O\left(r^{-n}\right)$ as $r\to\infty$. 
\end{lem}

Using these two lemmas, it follows that
\begin{lem}
\label{lem:InnerProductBound}
If $\mathbf{u}$ satisfies equation \eqref{eq:NoFlowAtInfty}, then on $\partial B_{r} \left(0 \right)$, 
\begin{equation}
p\left( \mathbf{u} , \mathbf{n} \right) - \omega \left(\mathbf{u}^{\perp},\mathbf{n} \right) = O\left(r^{-1}\right)\;\;\text{as}\;\; r\to\infty \, .
\end{equation}
\end{lem}

\begin{lem}
\label{lem:BoundaryTermZero}
If $\mathbf{u}$ satisfies equations 
\eqref{eq:RigidBodyMotion}, \eqref{eq:ForceCondition} and \eqref{eq:TorqueConditions} with $\mathbf{F}_{i} = \mathbf{0}$ and
$T_{i} = 0$, then
\begin{equation}
\int_{\Gamma_i} \left(\mathbf{u}, \mathbf{f} \right) ds_{\bx} = 0 \, .
\end{equation}
\end{lem}
\begin{proof}
\begin{equation}
\int_{\Gamma_i} \left(\mathbf{u}, \mathbf{f} \right) ds_{\bx} = \int_{\Gamma_{i}} \left( \mathbf{v}_{i}+\omega_{i}\left(\mathbf{x} - 
\mathbf{x}^{c}_i \right)^{\perp},\mathbf{f}\right) ds_{\bx} = -\left(\mathbf{v}_{i} , \mathbf{F}_{i} \right) -\omega_{i}T_{i} = 0 \, .
\end{equation}
\end{proof}

\begin{lem}
\label{lem:BoundaryVortEst}
If $\mathbf{u}$ satisfies equations \eqref{eq:RigidBodyMotion}, \eqref{eq:ForceCondition} 
and \eqref{eq:TorqueConditions} with $\mathbf{F}_{i} = \mathbf{0}$ and
$T_{i} = 0$, then
\begin{equation}
 \int_{\Gamma_i} p\left(\mathbf{u}, \mathbf{n}\right) - \omega \left(\mathbf{u}^{\perp},\mathbf{n} \right) ds_{\bx} = 
-4 \omega_{i}^{2} \left|D_{i} \right|
\end{equation}
\end{lem}
\begin{proof}
 On $\Gamma_i$, $e\left(\mathbf{u}\right) = 0$ and $\omega = 2\omega_i$. Using Lemma \ref{lem:BoundaryTermZero} and the
divergence theorem
\begin{equation}
 -\int_{\Gamma_i} \left(\mathbf{u}, \mathbf{f} \right) 
 + 2\omega_i \left(\mathbf{v}_{i}^{\perp}+\omega_{i}\left(\mathbf{x} - 
\mathbf{x}^{c}_i \right),\mathbf{n} \right) ds_{\bx} = -4 \omega_{i}^{2} \left|D_{i} \right|
\end{equation}
\end{proof}

\begin{lem} [adapted from \cite{UniquenessNollFinn}]
If $\mathbf{u}\left(\mathbf{x}\right)$
satisifies equations \eqref{eq:StokesFlowEq}, \eqref{eq:MassConservation},
\eqref{eq:RigidBodyMotion}, \eqref{eq:ForceCondition}, \eqref{eq:TorqueConditions}
and \eqref{eq:NoFlowAtInfty} with $\mathbf{F}_{i}=0$ and $T_{i}=0$, then 
\begin{align}
\lim_{R\to\infty}\int_{\partial B_{R}\left(0\right)} \left(\mathbf{u},\mathbf{f} \right) ds_{\bx}  & \to0 \, . \label{eq:ForceDecayatInfty}
\end{align}
\end{lem}
\begin{proof}
For large enough $R$, Lemma \ref{lem:BoundaryVortEst} yields
\begin{align}
\int_{E\cap B_{R}\left(0\right)} \omega ^2 dV & = 
\sum_{i=1}^{N} \int_{\Gamma_i} p\left(\mathbf{u}, \mathbf{n}\right) - \omega \left(\mathbf{u}^{\perp},\mathbf{n} \right) ds_{\bx} \nonumber \\
&- \int_{\partial B_{R}\left(0\right)} p\left(\mathbf{u}, \mathbf{n}\right) - \omega \left(\mathbf{u}^{\perp},\mathbf{n} \right) ds_{\bx} \\
& =-4\sum_{i=1}^{N}\omega_{i}^{2} \left|D_{i} \right| -\int_{\partial B_{R}\left(0\right)}p\left(\mathbf{u}, \mathbf{n}\right) - \omega \left(\mathbf{u}^{\perp},\mathbf{n} \right) ds_{\bx} \, \label{eq:VolVortEst}.
\end{align}
Using Lemma \ref{lem:InnerProductBound}, we conclude that 
\begin{equation}
\int_{E} \omega ^2 dV < \infty \, .
\end{equation}
Using Lemma \ref{lem:HarmonicProperty1}, we know that
\begin{equation}
\omega (r,\theta) = r^{-1} a_1 (\theta) + O\left(r^{-2}\right) \, .
\end{equation}
Integrating $\omega^{2}$ in the annulus $B = B_{r} \left(0\right) \cap B_{\bar{R}} \left(0\right)^{C}$, we get
\begin{equation}
\int_{B} \omega^2 dV = \log{\left( \frac{r}{\bar{R}} \right)} \int_{0}^{2\pi} a_{1}^2(\theta) d\theta + O\left(r^{-2}\right) \, .
\end{equation}
Since $\int_{B} \omega^2 dV$ is bounded, we conclude that $a_{1} \equiv 0$ and that 
$\omega = O\left(r^{-2}\right)$. 
Using Lemma \ref{lem:pressureVorticityBound}, we 
conclude that $p=O\left(r^{-2}\right)$. Thus
\begin{equation}
\int_{\partial B_{R}\left(0\right)} \left[ -p\left( \mathbf{u} , \mathbf{n} \right) + \omega \left(\mathbf{u}^{\perp},\mathbf{n} \right) \right] ds_{\bx} \to 0\;\; \text{as} \;\; R\to\infty \, .
\end{equation}   
From equation \eqref{eq:VolVortEst}, it follows that $\omega\equiv0$ in $E$. Using equation \eqref{eq:MassConservation}, we conclude
that $\Delta \mathbf{u} = 0$ in $E$. Using the estimate for $p$ and Lemma \ref{lem:growth}, we get $\mathbf{f} = O\left(r^{-2}\right)$.
Using this estimate and the decay condition in $\mathbf{u}$ at $\infty$, the result follows.
\end{proof}
The following lemma is a modification of the standard proof of uniqueness
for Stokes flow \cite{kim2005microhydrodynamics,pozrikidis1992boundary}.

\begin{lem}
\label{Uniqueness mobility} If $\mathbf{u}\left(\mathbf{x}\right)$
satisifies equations \eqref{eq:StokesFlowEq}, \eqref{eq:MassConservation},
\eqref{eq:RigidBodyMotion}, \eqref{eq:ForceCondition}, \eqref{eq:TorqueConditions}
and \eqref{eq:NoFlowAtInfty} with $\mathbf{F}_{i}=0$ and $T_{i}=0$,
then $\mathbf{u}\left(\mathbf{x}\right) \equiv 0$.
\end{lem}
\begin{proof}
Let $\left\langle \cdot,\cdot\right\rangle :\bR^{2\times2}\times\bR^{2\times2}$
be the Frobenius inner product. For large enough $R$,
\begin{align*}
\int_{E\cap B_{R}\left(0\right)}\left\langle e\left(\mathbf{u}\right),e\left(\mathbf{u}\right)\right\rangle dV & =\int_{E\cap B_{R}\left(0\right)}\left\langle D\mathbf{u},e\left(\mathbf{u}\right)\right\rangle dV\\
 & \hspace{-.25in} =\int_{\partial(E\cap B_{R}\left(0\right))}\left(\mathbf{u},e\left(\mathbf{u}\right)\cdot\mathbf{n}\right)ds_{\bx}-\frac{1}{2}\int_{E\cap B_{R}\left(0\right)}\left(\mathbf{u},\Delta\mathbf{u}\right)dV\\
 & \hspace{-.25in} =\int_{\partial(E\cap B_{R}\left(0\right))}\left(\mathbf{u},e\left(\mathbf{u}\right)\cdot\mathbf{n}\right)ds_{\bx}-\frac{1}{2}\int_{E\cap B_{R}\left(0\right)}\left(\mathbf{u},\nabla p\right)dV\\
 & \hspace{-.25in} =\frac{1}{2}\int_{\partial\left(E\cap B_{R}\left(0\right)\right)}\left(\mathbf{u},\left(-p\left[\begin{array}{cc}
1 & 0\\
0 & 1
\end{array}\right]+2e\left(\mathbf{u}\right)\right).\mathbf{n}\right)ds_{\bx}\quad \\
 & \hspace{-.25in} =-\frac{1}{2}\sum_{i=1}^{N}\int_{\Gamma_{i}}\left(\mathbf{u},\mathbf{f}\right)ds_{\bx}+\frac{1}{2}\int_{\partial B_{R}\left(0\right)}\left(\mathbf{u},\mathbf{f}\right)ds_{\bx}\\
 & \hspace{-.25in} =\frac{1}{2}\int_{\partial B_{R}\left(0\right)}\left(\mathbf{u},\mathbf{f}\right) ds_{\bx} \, .
\end{align*}
using \eqref{eq:MassConservation} and Lemma \ref{lem:BoundaryTermZero}.
Taking the limit as $R\to\infty$ in the above expression and using
equation \eqref{eq:ForceDecayatInfty},
we get 
\begin{equation}
e\left(\mathbf{u}\right)\equiv\left[\begin{array}{cc}
0 & 0\\
0 & 0
\end{array}\right] \quad \mathbf{x} \in E \, .
\end{equation}
Thus, $\mathbf{u}$ is a rigid body motion. 
However since $\mathbf{u}(\bx) \to\mathbf{0}$
as $\left|\bx\right| \to \infty$, we conclude that $\mathbf{u}\equiv \mathbf{0}$.
\end{proof}

We construct an integral representation for the mobility problem by direct analogy
with the elastance problem, with
the velocity $\mathbf{u}(\mathbf{x})$ playing the role of the 
potential and surface traction $\mathbf{f}$ playing the role of charge
in the elastance problem. (A rigid body has no interior strain or stress, 
with all the stress residing on the surface.) 
We first construct an ``incident" field which satisfies the net force and torque 
conditions on each rigid body but which does not correspond to a rigid body motion.
We then find a ``scattered" velocity induced by an additional force vector 
$\mu$ so that the total velocity will satisfy
\eqref{eq:RigidBodyMotion} but does not change the net
force and torque.  As in the elastance problem,
this can be thought of as a redistribution of the surface force.

\subsection{Mathematical preliminaries}

In the remainder of this paper, we will use the Einstein summation convention.
As above, we let $\gamma$ be a smooth closed curve in $\bR^{2}$ and we 
let $D^{\mp}$ denote the domains corresponding to the interior and exterior 
of $\gamma$. 
$\mathbf{n}_{\bx} = (n_{\bx,1}, n_{\bx,2})$ 
will be used to denoted the unit outward normal at
$\bx \in \gamma$ and $\mathbf{n}_{0} = (n_{0,1}, n_{0,2})$ to denote the
unit outward normal at $\bx_{0}\in\gamma$. 
We let $\bmu(\bx) = (\mu_{1}(\bx), \mu_{2}(\bx)):\gamma\to\bR^2$ 
be a continuous function. 

Following the treatment of \cite{kim2005microhydrodynamics, pozrikidis1992boundary},
the fundamental solution to the Stokes equations (the Stokeslet) in free space 
is given by
\begin{equation}
G_{i,j}\left(\mathbf{x},\mathbf{y}\right)=\frac{1}{4\pi}\left[ -\log\left|\mathbf{x}-\mathbf{y}\right| \delta_{ij} + \frac{\left( x_{i} - y_{i} \right) \left( x_{j} - y_{j}
 \right)} {\left| \mathbf{x} - \mathbf{y} \right|^2}\right] \quad i,j \in {1,2} \, .
\end{equation}
The Stokeslet allows us to express the velocity field 
$\mathbf{u} = \left(u_{1} , u_{2} \right)$ 
induced by 
a point force $\mathbf{f} = \left(f_{1}, f_{2} \right)$ in the form
\begin{equation}
u_{i} = G_{i,j} \left(\mathbf{x},\mathbf{y}\right) f_j \, .
\end{equation}
The single layer Stokes potential is the velocity induced
by a surface force on a boundary $\gamma$:
\begin{equation}
\mathcal{S}_{\gamma}\mathbf{\bmu}\left(\mathbf{x}\right)_{i}=\int_{\gamma} G_{i,j} \left(\mathbf{x},\mathbf{y}\right)\mu_{j}\left(\mathbf{y}\right)ds_{\by} \quad
{\rm for}\ i=1,2 \, .
\label{stokesslpdef}
\end{equation}

\begin{lem}
Let 
$\mathcal{S}_{\gamma}\mathbf{\bmu}(\bx)$ denote a single layer Stokes potential 
of the form \eqref{stokesslpdef}. Then,
$\mathcal{S}_{\gamma}\bmu\left(\mathbf{x}\right)$ satisfies the Stokes equations
in $\bR^{2}\backslash\gamma$ and
$\mathcal{S}_{\gamma}\bmu(\mathbf{x})$ is continuous
in $\bR^{2}$.
Moreover, if we let
$\mathbf{f}(\mathbf{x}_0)$ denote the surface traction on $\gamma$ corresponding to 
the velocity field $\mathcal{S}_{\gamma}\bmu\left(\mathbf{x}\right)$, then
\begin{equation}
\lim_{\substack{\mathbf{x}\to\mathbf{x}_{0}\\ \mathbf{x}\in D^{\pm}}} f_{i,\pm} 
\left(\mathbf{x}_0 \right)=\mp\frac{1}{2}\mu_i\left(\mathbf{x}_{0}\right)+n_{0,k}\oint_{\gamma} \mathbf{T}_{i,j,k} \left(\mathbf{x}_{0},\mathbf{y}\right) \mu_{j} \left(\mathbf{y}\right)ds_{\by} \label{SLnormaljumpStokes}
\end{equation}
 where $\mathbf{T}_{i,j,k} \left(\mathbf{x},\mathbf{y}\right) $ is the 
stresslet corresponding to the flow given by
\begin{equation}
\mathbf{T}_{i,j,k} \left(\mathbf{x},\mathbf{y}\right) = -\frac{1}{\pi} \frac{\left( x_{i} - y_{i} \right) \left( x_{j} - y_{j} \right) \left( x_{k} - y_{k} \right)}{\left| \mathbf{x} - \mathbf{y} \right|^{4}} \, .
\end{equation} 
The notation $\oint_\gamma$ is used, as above, to denote the principal value 
integral.
The net force and torque on the domain are given by 
\begin{equation}
\int_{\gamma}\mathbf{f}_{+} ds_{\bx} =
-\int_{\gamma}\mathbf{\bmu}\left(\bx\right)ds_{\bx} \, ,\quad
\int_{\gamma}\mathbf{f}_{-} \, ds_{\bx} = \mathbf{0}
\label{eq:SLGaussLaw1StokesForce}
\end{equation}
and
\begin{align}
\int_{\gamma}\left( \left(\bx-
\mathbf{x}^{c}\right)^{\perp} ,\mathbf{f} \right)_{+} ds_{\bx}  & =
-\int_{\gamma}\left( \left(\bx-\mathbf{x}^{c}\right)^{\perp} ,
\bmu \right)ds_{\bx}\label{eq:SLGaussLaw1StokesTorque} \, ,\\
\int_{\gamma}\left( \left(\bx-\mathbf{x}^{c}\right)^{\perp} ,
\mathbf{f} \right)_{-} ds_{\bx}  & =0 \label{eq:SLGaussLaw12StokesTorque}\, .
\end{align}
If
$\omega$ is a closed curve in $D^{+}$, then
\begin{align}
\int_{\omega} \mathbf{f} \, ds_{\bx} & = \mathbf{0} \label{eq:SLGaussLaw2StokesForce}
\\
\int_{\omega} \left( \left(\bx-
\mathbf{x}^{c}\right)^{\perp} ,\mathbf{f} \right) ds_{\bx} & =0 \, .
\label{eq:SLGaussLaw2StokesTorque}
\end{align}
Finally,
\begin{equation}
\left|\mathcal{S}_{\gamma}\mathbf{\bmu}\left(\mathbf{x}\right)+
\frac{1}{4\pi} \left[\log\left(\mathbf{x}\right)\left[\begin{array}{cc}
1 & 0\\
0 & 1
\end{array}\right] - \frac{\mathbf{R}}{\left|x\right|^2} \right] \int_{\gamma}
\bmu\left(\mathbf{y}\right)ds_{\by}\right|
\to 0
\label{eq:SLGrowthAtInftyStokes}
\end{equation}
as $\left|\mathbf{x}\right|\to\infty$, where
\begin{equation}
\mathbf{R} =
\left[\begin{array}{cc} 
x_{1}^2 & x_{1} x_{2}\\
x_{1} x_{2} & x_{2}^{2}
\end{array}\right] \, .
\end{equation}
\end{lem}

The double layer Stokes potential is the velocity field due to a surface
density of stresslets on the curve:
\begin{equation}
\mathcal{D}_{\gamma}\bmu\left(\mathbf{x}\right)_{i}=\int_{\gamma}
\mathbf{T}_{j,i,k}\left(\mathbf{y},\mathbf{x} \right) \mathbf{\mu}_{j} \left(\mathbf{y}\right) n_{\by,k} \, ds_{\by} \, .
\label{stokesdlpdef}
\end{equation}

\begin{lem}
Let $\mathcal{D}_{\gamma}\bmu\left(\mathbf{x}\right)$ denote a double layer Stokes
potential of the form \eqref{stokesdlpdef}. Then,
$\mathcal{D}_{\gamma}{\bmu}\left(\mathbf{x}\right)$ satisfies the
Stokes equation in $\bR^{2}\backslash\gamma$ and the jump relations:
\begin{equation}
\lim_{\substack{\mathbf{x}\to\mathbf{x}_{0}\\
\mathbf{x}\in D^{\pm}}} \mathcal{D}_{\gamma}\mathbf{\mu}_{i} =
\pm\frac{1}{2}\mathbf{\mu}_{i}
\left(\mathbf{x}_{0}\right)+\oint_{\gamma}\mathbf{T}_{j,i,k} \left(\mathbf{y}, \mathbf{x}_{0}\right) {\mu_j}\left(\mathbf{y}\right) n_{\by,k}  \, ds_{\by} \, . \label{eq:DLP0Stokes}
\end{equation}
Furthermore,
\begin{align}
\int_{\gamma}\mathbf{T}_{i,j,k}\left(\mathbf{y},\mathbf{x} \right) n_{\by,k} \, ds_{\by} & =
\begin{cases}
-\delta_{ij} & \quad\mathbf{x}\in D\\
0 & \quad\mathbf{x}\in E
\end{cases} \, ,\label{eq:DLP1Stokes}\\
\oint_{\gamma}\mathbf{T}_{i,j,k}\left(\mathbf{y},\mathbf{x} \right) n_{\by,k} \, ds_{\by} & =
-\frac{\delta_{ij}}{2}\quad\mathbf{x}\in\gamma \, .\label{eq:DLP2Stokes}
\end{align}
Letting $\epsilon_{ilj}$ be the standard Levi-Civita symbol, 
\begin{align}
\int_{\gamma}\epsilon_{ilm}y_{l} \mathbf{T}_{m,j,k}\left(\mathbf{y},\mathbf{x} \right) n_{\by,k} \, ds_{\by} 
& =\begin{cases}
-\epsilon_{ilj}x_{l} & \quad\mathbf{x}\in D\\
0 & \quad\mathbf{x}\in E
\end{cases} \, , \label{eq:DLP1StokesTorque}\\
\oint_{\gamma}\epsilon_{ilm} y_{l} \mathbf{T}_{m,j,k}\left(\mathbf{y},\mathbf{x} \right) n_{\by,k} \, ds_{\by} & 
=-\frac{\epsilon_{ilj}x_{l}}{2}\quad\mathbf{x}\in\gamma \, . \label{eq:DLP2StokesTorque}
\end{align}
Finally,
\begin{equation}
\left|\mathcal{D}_{\gamma}\mathbf{\bmu}\left(\mathbf{x}\right)\right|\to0\quad\mbox{as }\left|\mathbf{x}\right|\to\infty \, . \label{eq:DLGrowthAtInftyStokes}
\end{equation}
\end{lem}

\subsection{Applying the net force and torque as an incident field \label{sec:intmob}}

We construct a velocity field $\mathbf{u}_{inc}\left(\mathbf{x}\right)$ 
in the exterior domain $E$, 
due to set of surface force densities 
$\left\{ {\brho}_{j}\left(\mathbf{x}\right)\right\} _{j=1}^{N}$
on the boundaries $\left\{ \Gamma_{j}\right\} _{j=1}^{N}$, 
which satisfies the force and torque constraints 
\eqref{eq:ForceCondition} and \eqref{eq:TorqueConditions}.
Each ${\brho}_j$ is a vector density 
${\brho}_j =\left(\rho_{1,j},\rho_{2,j} \right)$. 
Letting ${\brho} \left(\mathbf{x}\right)=
\left(\rho_{1,1}\left(\mathbf{x}\right),\rho_{2,1}\left(\mathbf{x}\right)
\ldots\rho_{1,N}\left(\mathbf{x}\right), \rho_{2,N}\left(\mathbf{x}\right)\right)$,
we define
\begin{align}
\mathbf{u}_{inc}\left(\mathbf{x}\right) & 
=\mathcal{S}_{\Gamma}{\brho}\left(\mathbf{x}\right) \, ,
\label{eq:IncRepresentationStokes}
\end{align}
where $\mathcal{S}_{\Gamma}$
is the operator given by
\begin{align}
\mathcal{S}_{\Gamma}{\brho}\left(\mathbf{x}\right)_{i} & 
= \sum_{j=1}^{N} \mathcal{S}_{\Gamma_{j}}{\brho}_{j}\left(\mathbf{x}\right)_{i}=
\sum_{j=1}^{N}\int_{\Gamma_{j}}G_{i,k}\left(\mathbf{x},\mathbf{y}\right)\mathbf{\rho}_{k,j}\left(\mathbf{y}\right)ds_{\by} \quad i=1,2 \, .
\end{align}
If we now let
$\mathbf{f}_{j}$ denote the surface force on $\Gamma_j$ corresponding to 
the velocity field $\mathbf{u}_{inc}$ and make use of 
equations \eqref{eq:SLGaussLaw1StokesForce}
and \eqref{eq:SLGaussLaw2StokesForce}, we obtain
\begin{align}
\mathbf{F}_{j} = -\int_{\Gamma_{j}} \mathbf{f}_{j} \, ds_{\bx} 
= \int_{\Gamma_{j}}{\brho}_{j}\left(\mathbf{x}\right)ds_{\bx}
\quad{\rm for}\ j=1,2,\ldots N \, .
\label{eq:ForceImposition}
\end{align}
Using equations \eqref{eq:SLGaussLaw1StokesTorque},  and 
\eqref{eq:SLGaussLaw2StokesTorque}, we obtain
\begin{align}
T_{j} = -\int_{\Gamma_{j}} \left( \left(\mathbf{x} - \mathbf{x}^{c}_{j} \right)^{\perp}, \mathbf{f}_j \right) ds_{\bx} = 
\int_{\Gamma_j} \left( \left(\mathbf{x} - \mathbf{x}^{c}_j \right)^{\perp}, {\brho}_j \right) ds_{\bx} \, . \label{eq:TorqueImposition}
\end{align}
Thus, any choice of ${\brho}_{j}\left(\mathbf{x}\right)$ which satisfies
equations \eqref{eq:ForceImposition} and \eqref{eq:TorqueImposition} will define
an incident field that enforces
the desired force and torque conditions. 
We will use the simple formula
\begin{equation}
{\brho}_{j}\left(\mathbf{x}\right)=\frac{\mathbf{F}_{j}}{\left|\Gamma_{j}\right|} 
+ T_{j}\frac{\left(\mathbf{x}-\mathbf{x}^{c}_j \right)^{\perp}}
{W_{j}} \, ,
\label{eq:rho}
\end{equation}
where $\left|\Gamma_{j}\right|$ is the length of $\Gamma_{j}$
and
$W_{j} = \int_{\Gamma_{j}} \left| \mathbf{x} - \mathbf{x}^{c}_j \right|^2 ds_{\bx}$.

\subsection{The scattered field}

We now seek
a ``scattered'' velocity field $\mathbf{u}_{sc}\left(\mathbf{x}\right)$ induced
by unknown
surface force densities $\left\{ {\bmu}_{j}\left(\mathbf{x}\right)\right\} _{j=1}^{n}$
on the boundaries $\left\{ \Gamma_{j}\right\} _{j=1}^{n}$. 
Each ${\bmu}_j$ is a vector density 
${\bmu}_j =\left(\mu_{1,j},\mu_{2,j} \right)$. 
These densities correspond to a redistribution of 
surface forces that will be used to enforce the rigid body boundary conditions
without affecting the net force and torque.
We let
\[ {\bmu} \left(\mathbf{x}\right)=\left(\mu_{1,1}\left(\mathbf{x}\right),\mu_{2,1}\left(\mathbf{x}\right)
\ldots\mu_{1,N}\left(\mathbf{x}\right), \mu_{2,N}\left(\mathbf{x}\right)\right) 
\]
and define
\begin{align}
\mathbf{u}_{sc}\left(\mathbf{x}\right) & =
\mathcal{S}_{\Gamma}{\bmu}
\left(\mathbf{x}\right) \, . \label{eq:ScRepresenationStokes}
\end{align}
To ensure that no additional net forces or torques are introduced
on the surfaces $\Gamma_i$, we need to impose $3N$ integral constraints on
${\bmu}\left(\mathbf{x}\right)$, namely
\begin{align}
\int_{\Gamma_{j}} \mu_{i,j}\left(\mathbf{x}\right)ds_{\bx} & =0 \, , \\
\int_{\Gamma_{j}} \left(\left( \mathbf{x} - \mathbf{x}^{c}_j \right)^{\perp},\boldsymbol{\mu}_{j} \right) ds_{\bx} & = 0 \, .
\end{align}
The total velocity field is given by
\begin{align}
\mathbf{u}\left(\mathbf{x}\right) & =\mathbf{u}_{inc}\left(\mathbf{x}\right)+\mathbf{u}_{sc}\left(\mathbf{x}\right)
=\mathcal{S}_{\Gamma}\left({\bmu}\left(\mathbf{x}\right)+{\brho}\left(\mathbf{x}\right)\right) \, .
\label{eq:TotRepresentationStokes}
\end{align}

\subsection{Reformulation as an interior boundary value problem}

The function 
$\mathbf{u}(\bx) = \mathcal{S}_{\Gamma}({\bmu}(\bx)+{\brho}(\bx))$
also represents the velocity field inside
the rigid bodies. 
Since there is no internal stress in a rigid body, 
the stress tensor $\boldsymbol{\sigma}$
must be identically zero within $D_i$. Thus we will seek to impose
\begin{equation}
\mathbf{f}_{-} = \left(\boldsymbol{\sigma}\cdot \mathbf{n}\right)_{-} \equiv0
\end{equation}
for $\mathbf{x}\in\Gamma$. 
Using equation \eqref{SLnormaljumpStokes},
we obtain the following Fredholm integral equation of the second kind:
\begin{align}
\left(\frac{1}{2}\mathbf{I}+\mathcal{K}\right)\bmu\left(\mathbf{x}\right) & 
=-\left(\frac{1}{2}\mathbf{I}+\mathcal{K}\right)\brho\left(\mathbf{x}\right)\quad\mathbf{x}\in\Gamma\label{eq:IntegralEqStokes}
\end{align}
which we subject to the constraints
\begin{align}
\int_{\Gamma_{j}} \mu_{i,j}\left(\mathbf{x}\right)ds_{\bx} & =0 \, , 
\label{eq:IntegralEqStokesForceConstraint}\\ 
\int_{\Gamma_{j}} \left( ( \mathbf{x} - \mathbf{x}^{c}_j)^{\perp},\bmu_{j} \right) ds_{\bx} & = 0 \, , 
\label{eq:IntegralEqStokesTorqueConstraint}
\end{align}
where
\[
\mathbf{I} =\left(\begin{array}{cccc}
\mathbf{I}_{1}\\
 & \mathbf{I}_{2}\\
 &  & \ddots\\
 &  &  & \mathbf{I}_{N}
\end{array}\right) 
\] 
and
\[
\mathcal{K} =\left(\begin{array}{cccc}
\mathcal{K}_{1,1} & \mathcal{K}_{1,2} & \ldots & \mathcal{K}_{1,N}\\
\mathcal{K}_{2,1} & \mathcal{K}_{2,2} & \ldots & K_{2,N}\\
\vdots & \vdots & \ddots & \vdots\\
\mathcal{K}_{N,1} & \mathcal{K}_{N,2} &  & \mathcal{K}_{N,N}
\end{array}\right).
\]
Here,
$\mathbf{I}_{i}: C^{0,\alpha}\left(\Gamma_{i}\right) \times 
C^{0,\alpha}\left(\Gamma_{i}\right) \to C^{0,\alpha}\left(\Gamma_{i} \right) 
\times C^{0,\alpha}\left(\Gamma_{i} \right)$ 
is the identity map, 
and $\mathcal{K}_{i,j}:C^{0,\alpha}\left(\Gamma_{j}\right) \times 
C^{0,\alpha}\left(\Gamma_{j}\right) 
\to C^{0,\alpha}\left(\Gamma_{i}\right) \times C^{0,\alpha}\left(\Gamma_{i}\right)$ 
is the operator given by
\[
\left(\mathcal{K}_{i,j}\brho \right)_{k} = n_{\bx,l}\int_{\Gamma_{j}} \mathbf{T}_{k,m,l}\left(\mathbf{x}, \mathbf{y} \right) 
\rho_{m}\left(\mathbf{y}\right) ds_{\by} \quad \mathbf{x}\in\Gamma_{i}
\]
for $i \neq j$ and 
\[
\left(\mathcal{K}_{i,i}\brho \right)_{k} = n_{\bx,l}
\oint_{\Gamma_{i}} \mathbf{T}_{k,m,l}\left(\mathbf{x}, \mathbf{y} \right) 
\rho_{m}\left(\mathbf{y}\right) ds_{\by} \quad \mathbf{x}\in\Gamma_{i} \, .
\]

\begin{thm}
Let $\mathbf{u}\left(\mathbf{x}\right)$ be the total velocity,
defined in \eqref{eq:TotRepresentationStokes}.
If ${\bmu}\left(\mathbf{x}\right)$ solves equation \eqref{eq:IntegralEqStokes},
together with the constraints
\eqref{eq:IntegralEqStokesForceConstraint} and 
\eqref{eq:IntegralEqStokesTorqueConstraint}, 
then $\mathbf{u}\left(\mathbf{x}\right)$ solves the mobility problem.
\end{thm}
\begin{proof}
$\mathbf{u}\left(\mathbf{x}\right)$ clearly satisfies the Stokes equations in
$E$ by construction.
Using equations \eqref{eq:SLGaussLaw1StokesForce}, \eqref{eq:SLGaussLaw1StokesTorque},
\eqref{eq:SLGaussLaw2StokesForce} and \eqref{eq:SLGaussLaw2StokesTorque},
the choice of ${\brho}$ in equation \eqref{eq:rho}, and the constraints
\eqref{eq:IntegralEqStokesForceConstraint} and 
\eqref{eq:IntegralEqStokesTorqueConstraint}, 
we see that $\mathbf{u}\left(\mathbf{x}\right)$ satisfies the net force and 
torque conditions 
\eqref{eq:ForceCondition} and \eqref{eq:TorqueConditions}.
Furthermore, from \eqref{eq:NoNetForce}
and \eqref{eq:SLGrowthAtInftyStokes}, it follows that 
$\left|\mathbf{u}\left(\mathbf{x}\right)\right|\to0$
as $\left|\mathbf{x}\right|\to\infty$. 
Since $\mathbf{u}\left(\mathbf{x}\right)$ solves the Stokes equations
in $D_{i}$ and satisfies $\mathbf{f}_{-}\equiv0$
on $\Gamma_{i}$, $\mathbf{u}$ must be a rigid body motion. 
By the continuity of the single layer potential, 
$\mathbf{u}$ must define a rigid body motion from the exterior as well.
\end{proof}

\subsection{Existence of solutions \label{sec:extmob}}

It is well-known that $\frac{1}{2}\mathbf{I}+\mathcal{K}$ has a $3N$-dimensional
null space \cite{pozrikidis1992boundary}.
It follows from the Fredholm alternative that 
$\left(\frac{1}{2}\mathbf{I}+\mathcal{K}\right)\mathbb{\mu}=g$
has an $3N$ dimensional space of solutions, so long as $g$ is in the
range of the operator $\frac{1}{2}\mathbf{I}+\mathcal{K}$. 
From \eqref{eq:IntegralEqStokes}, this is clearly the case, and
the purpose of the 
$3N$ integral constraints is to select the particular solution
that doesn't alter the net forces and torques.
As for the elastance problem, however, we do not wish to solve
a rectangular linear system. If we discretize
the integral equation using Nystr\"{o}m quadrature, with $M$ points
on $\Gamma$, we would have to solve a $(2M+3N)\times 2M$ linear system
to obtain the desired solution. Instead, we propose to solve
the integral equation
\begin{align}
\frac{1}{2}{\bmu}_{i}\left(\mathbf{x}\right)+\sum_{j=1}^N
\mathcal{K}_{i,j}\bmu_{j}\left(\mathbf{x}\right)
+\int_{\Gamma_{i}}\bmu_{i} \left(\mathbf{y}\right)ds_{\by} \hspace{1.2in} \nonumber \\
+ \left(\mathbf{x} - \mathbf{x}_c \right)^{\perp}
\int_{\Gamma_{i}} \left( \left(\mathbf{y} - \mathbf{x}^{c}_i \right)^{\perp}, \bmu_{i} \left(\mathbf{y} \right) \right) ds_{\by} \hspace{0.6in} \nonumber \\
 = - \frac{1}{2}\brho_{i}\left(\mathbf{x}\right)-\sum_{j=1}^N
\mathcal{K}_{i,j}\brho_{j}\left(\mathbf{x}\right)
\quad\mathbf{x}\in\Gamma_{i}\label{eq:ModrepStokes}
\end{align}
or
\begin{align}
\left(\frac{1}{2}\mathbf{I}+\mathcal{K}+\mathbf{L}\right)\bmu & =-\left(\frac{1}{2}\mathbf{I}+\mathcal{K}\right)\brho \, ,
\label{eq:ModRep2Stokes}
\end{align}
where
\begin{align}
\mathbf{L} & =\left(\begin{array}{cccc}
\mathbf{L}_{1}\\
 & \mathbf{L}_{2}\\
 &  & \ddots\\
 &  &  & \mathbf{L}_{N}
\end{array}\right)
\end{align}
with
$\mathbf{L}_{i}$ 
defined by 
\begin{equation}
\mathbf{L}_{i}\bmu_{i}\left(\mathbf{x}\right)=\int_{\Gamma_{i}}\bmu_{i}\left(\mathbf{y}\right)ds_{\by}
+ \left(\mathbf{x} - \mathbf{x}^{c}_i \right)^{\perp}
\int_{\Gamma_{i}} \left( \left(\mathbf{y} - \mathbf{x}^{c}_i \right)^{\perp}, \bmu_{i} \left(\mathbf{y} \right) \right) ds_{\by}  \, .
\end{equation}

The following lemma shows that solving \eqref{eq:ModRep2Stokes}
is equivalent to solving \eqref{eq:IntegralEqStokes} with the 
constraints \eqref{eq:IntegralEqStokesForceConstraint} and 
\eqref{eq:IntegralEqStokesTorqueConstraint}. 

\begin{lem}
If ${\bmu}$ solves \eqref{eq:ModRep2Stokes}, then it solves 
\eqref{eq:IntegralEqStokes}, \eqref{eq:IntegralEqStokesForceConstraint} and 
\eqref{eq:IntegralEqStokesTorqueConstraint} \label{Lem:equivalenceStokes}.
\end{lem}

\begin{proof}
Using equation \eqref{eq:DLP1StokesTorque}, we see that
$\int_{\Gamma_{i}} \left(\left(\mathbf{x} - \mathbf{x}^{c}_i \right)^{\perp}, \mathcal {K}_{i,j}\bmu_{j}\left(\mathbf{x}\right)\right) ds_{\bx}  = 0$.\\
Similarly, using equation \eqref{eq:DLP2StokesTorque}, we see that
\begin{equation}
\int_{\Gamma_{i}} \left(\left(\mathbf{x} - \mathbf{x}^{c}_i \right)^{\perp}, \mathcal {K}_{i,i}\bmu_{i}\left(\mathbf{x}\right) \right) ds_{\bx} 
= -\frac{1}{2}\int_{\Gamma_{i}} \left( \left(\mathbf{x} - \mathbf{x}^{c}_i \right)^{\perp}, \bmu_{i}\left(\mathbf{x}\right) \right) ds_{\bx} \, .
\end{equation}
Since $\bx^{c}_i$ is the centroid of $\Gamma_{i}$, $\int_{\Gamma_{i}} \left(\mathbf{x} - \mathbf{x}^{c}_i \right)^{\perp} ds_{\bx} = 0$.
Taking the inner product of \eqref{eq:ModRep2Stokes} with 
$\left(\mathbf{x} - \mathbf{x}^{c}_i \right)^{\perp}$, integrating
the expression over $\Gamma_{i}$, and using the equations above, we obtain
\begin{equation}
\left(\int_{\Gamma_{i}} \left| \mathbf{x} - \mathbf{x}^{c}_i \right|^2 ds_{\bx} \right) \int_{\Gamma_{i}} 
\left( \left(\mathbf{y} - \mathbf{x}^{c}_i \right)^{\perp}, \bmu_{i} \left(\mathbf{y} \right) \right) ds_{\by} = 0 \, .
\end{equation}
From \eqref{eq:DLP1Stokes}, 
we observe that 
$\int_{\Gamma_{i}}\mathcal{K}_{i,j}\bmu_{j}
\left(\mathbf{x}\right)ds_{\bx}=0$ for $j\neq i$. 
Furthermore, switching the order of integration in
$\int_{\Gamma_{i}}\mathcal{K}_{i,i}\bmu_{i}\left(\mathbf{x}\right)ds_{\bx}$ and using property \eqref{eq:DLP2Stokes} of the double layer potential, 
we find that
\begin{align}
\int_{\Gamma_{i}}K_{i,i}\bmu_{i}\left(\mathbf{x}\right)ds_{\bx} 
 & =-\frac{1}{2}\int_{\Gamma_{i}}\bmu_{i}\left(\bx\right)ds_{\bx} \, .
\end{align}
Integrating the expression \eqref{eq:ModRep2Stokes} on $\Gamma_{i}$ 
and using the fact that 
\[ \int_{\Gamma_{i}} 
\left( \left(\mathbf{y} - \mathbf{x}^{c}_i \right)^{\perp}, \bmu_{i} \left(\mathbf{y} \right) \right) ds_{\by} = 0,
\] 
we may conclude that
\begin{equation}
\left|\Gamma_{i}\right|\int_{\Gamma_{i}}\bmu_{i}\left(\mathbf{x}\right)ds_{\bx}=0 \, .
\end{equation}
Thus, ${\bmu}$ satisfies the integral constraints 
\eqref{eq:IntegralEqStokesForceConstraint} and
\eqref{eq:IntegralEqStokesTorqueConstraint}, implying that \\
$\mathbf{L}_{i}\bmu_{i}\left(\mathbf{x}\right)=0$ 
and that 
\begin{equation}
\left(\frac{1}{2}\mathbf{I}+\mathcal{K}+\mathbf{L}\right)\bmu
=\left(\frac{1}{2}\mathbf{I}+\mathcal{K}\right)\bmu=-\left(\frac{1}{2}\mathbf{I}+\mathcal{K}\right)\brho \, .
\end{equation}
\end{proof}

The following lemma shows that the operator $\frac{1}{2}\mathbf{I}+\mathcal{K}+\mathbf{L}$ has no null space.
\begin{lem}
 The operator $\frac{1}{2}\mathbf{I}+\mathcal{K}+\mathbf{L}$ is injective.
\end{lem}
\begin{proof}
 Let $\bmu \in \mathcal{N}\left(\frac{1}{2}\mathbf{I} + \mathcal{K} +\mathbf{L} \right)$, i.e. it solves
 $\left(\frac{1}{2}\mathbf{I}+\mathcal{K}+\mathbf{L}\right)\bmu = 0$. 
 Following the proof of Lemma \ref{Lem:equivalenceStokes},
 we conclude that $\bmu$ satisfies the force and torque constraints given by equations 
 \eqref{eq:IntegralEqStokesForceConstraint} and \eqref{eq:IntegralEqStokesTorqueConstraint}.
 Thus $\mathbf{L}\bmu = 0$ and $\left(\frac{1}{2}\mathbf{I} + \mathcal{K}\right)\bmu = 0$. 
 Let $\mathbf{u}=\mathcal{S}_{\Gamma} \bmu$.  
 Let $\mathbf{f}_{-}$ and $\mathbf{f}_{+}$ denote the interior 
 and exterior limits of the surface traction corresponding to 
 the velocity field $\mathbf{u}$, respectively. 
 From the properties of the Stokes single layer potential 
 \begin{equation}
 \mathbf{f}_{-} = \left(\frac{1}{2}\mathbf{I}+\mathcal{K} \right)\bmu = 0 \, .
 \end{equation}
 By uniqueness of solutions to interior surface traction problem, 
 we conclude that $\mathbf{u}$ is a rigid body motion on each boundary
 component. Thus, $\mathbf{u}$ solves the mobility problem with $\mathbf{F}_{i} = \mathbf{0}$ and $T_{i}=0$.  
 By uniqueness of
 solutions to the mobility problem, we conclude that $\mathbf{u}\equiv 0$ in $E$. Hence, 
 $\mathbf{f}_{+} =0$. From the properties of the Stokes single layer,
 \begin{equation}
  \bmu = \mathbf{f}_{-} - \mathbf{f}_{+}= \mathbf{0} \, .
 \end{equation}
 Therefore, $\mathcal{N} \left(\frac{1}{2}\mathbf{I} + \mathcal{K} + \mathbf{L}\right) = \left \{ 0 \right \}$.
\end{proof}
By the Fredholm alternative, therefore, \eqref{eq:ModRep2Stokes} has a unique
solution $\mathbf{\mu}$.

\section{Numerical Examples\label{sec:num}}
The fact that the capacitance and elastance problems are inverses of each other,
and that completely different techniques can be used for their solution,
permits a robust test of the performance of our method in 
arbitrary geometry (without an exact reference solution).
The same is true for the resistance and mobility problems.

\subsection{The elastance problem \label{sec:4.1}}

Suppose now that we solve the capacitance problem
discussed in section \ref{sec:elast} using known techniques (see, \cite{mikhlin1964integral},
for example). That is,
given prescribed potentials $\phi_{j}$ on a collection of perfect conductors with
boundaries $\Gamma_{j}$, we may obtain the charges induced on each conductor.
We can then solve the elastance problem with these charges as input, using
the representation in section \ref{sec:elast}, and verify that the corresponding
potentials are those used in the original capacitance problem setup.
We emphasize that 
the integral equations used for the capacitance and elastance problems are 
not inverses of each other, so this provides a nontrivial test of accuracy.

More precisely, following the discussion 
in section \ref{sec:elast}, we consider the domain exterior 
to $N$ perfect conductors $D_{i}$ whose boundaries are given by $\Gamma_{i}$.
We prescribe potentials $\phi_{i}$ on the boundaries $\Gamma_{i}$ and 
solve the capacitance problem to obtain the net charge $q_{i}$ on $\Gamma_{i}$
and also the potential at $\infty$, $u_{\infty} = \lim_{\left|\bx \right| \to \infty} u\left(\bx \right)$.

We use these charges
as input for the elastance problem and to compute the potentials induced on the conductors, letting
$\sigma_{i,el}$
denote the uniformly distributed charge defined in terms of $q_{i}$,
as in section \ref{sec:incelast}.
$\mu_{i,el}$, as before, represents the unknown density on $\Gamma_{i}$ for the 
elastance problem and 
\begin{align}
u(\bx) & =u_{inc}(\bx)+u_{sc}(\bx) + u_{\infty}=
S_{\Gamma}(\mu_{el}+\sigma_{el})(\bx) + u_{\infty} \, . \label{eq:TotRepresentation2}
\end{align}
We then solve
\begin{equation}
\left(\frac{1}{2}I+K+L\right)\mu_{el}  = 
-\left(\frac{1}{2}I+K\right)\sigma _{el}, \label{eq:ModRep3}
\end{equation}
where $u_{\infty}$ is the potential
at $\infty$ computed in the capacitance problem, and the operators $K$ and $L$ are described in
section \ref{sec:extelast}. This is a small modification
of the representation presented in section \ref{sec:incelast} to account for the potential
at $\infty$. 
After solving for $\mu_{el}$, we may check the accuracy with which 
$u$ in equation \ref{eq:TotRepresentation2}
equals the potential $\phi_{j}$ on $\Gamma_{j}$, the potentials
prescribed in the original capacitance problem.

\subsubsection{Two disc test \label{sec:6.1.1}}
We first consider the case of two unit discs
separated by a 
distance $d$. This is useful because the exact solution is known and because we
wish to study the physical ill-conditioning of the problem as 
$d\to 0$.
\begin{figure}[H]
\begin{center}
 \includegraphics[width=6cm]{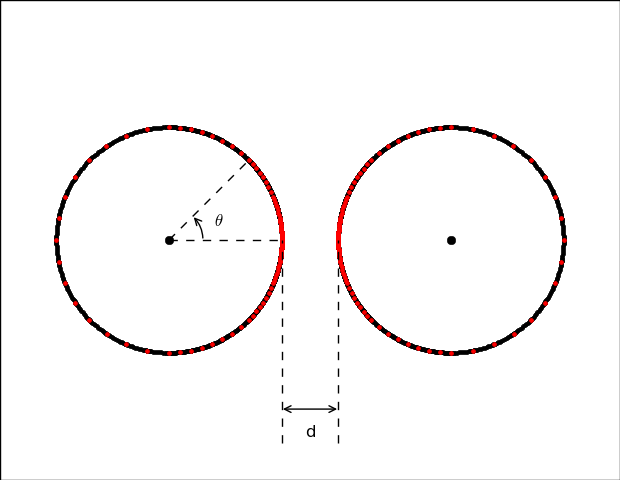}
 \end{center}
 \caption[Two discs problem - Elastance]{Discretization of the discs for Elastance example.
\label{Fig:twocd_el}}
\end{figure}

In the context of Fig. \ref{Fig:twocd_el}, we set $u|_{\Gamma_{i}} = 
\phi_{i}$ for $i=1,2$ (where $i=1$ corresponds to the left disc), with $\phi_{1} = 0.209$ and $\phi_{2} =-0.123$.
We consider $d=0.5,0.05,0.005$.

We use a Nystr\"{o}m discretization, based on subdivision of the boundary 
into panels, with Gauss-Legendre nodes given on each panel.
Let $\mathbf{s}_{i,j,l}$ denote the $j$th node 
on the $i$th panel on boundary component $l$. Let
$\sigma_{i,j,l,el}, \mu_{i,j,l,el}$ denote the density evaluated at $\mathbf{s}_{i,j,l}$.
We use a recently developed quadrature scheme, denoted by GLQBX 
(global + local quadrature by expansion) 
\cite{rachh15,GLQBX} for evaluating the layer potential 
$K$ in equation
\eqref{eq:ModRep3}.
This scheme is a robust extension of the QBX method of
\cite{QBX}, guaranteed to yield high order accuracy even when boundaries are close-to-touching.
We use an iterative GMRES-based solver to obtain 
to obtain $\mu_{el}$, and iterate to a relative
residue of $10^{-6}$. As $d\to 0$,
the problem becomes physically ill-conditioned, requiring an increasing
number of iterations. 
To improve the rate of convergence, we use an $L^2$-based rescaling of
the unknowns \cite{bremer2010efficient}.
That is, we use  
$\mu_{el}^{scale} = \mu_{i,j,l,el} \sqrt{r_{i}}$ as unknowns, so that the discrete 2-norm approximates
the $L^2$ norm where $r_{i}$ is the length of panel $i$.

We iterate the following discretized linear system
\begin{equation}
D \left(\frac{1}{2}I+\tilde{K}+\tilde{L}\right) D^{-1} \mu_{el}^{scale} =
-D \left(\frac{1}{2}I+\tilde{K}\right)\sigma_{el} \, ,
\end{equation}
where $\tilde{K}$ and $\tilde{L}$ are discretized versions of $K$ and $L$ 
respectively, $D$ is the diagonal operator described given by $D\mu_{el} = \mu_{el}^{scale}$. 

In Fig.  \ref{Fig:sigma1el}, we plot
the net charge density $\sigma_{1,el} +
\mu_{1,el}$ for the three different values of $d$.
In Fig. \ref{Fig:contourplot}, we plot the potential using the 
off-surface evaluation method of \cite{barnett-2014}, whose development
initially led to QBX. (The option of off-surface evaluation has been 
incorporated into our QBX software.)

From symmetry considerations, 
$\sigma_{2,el} + \mu_{2,el} = -\left(\sigma_{1,el} + \mu_{1,el} \right)$.
\begin{figure}[H]
\begin{center}
 \includegraphics[width=6cm]{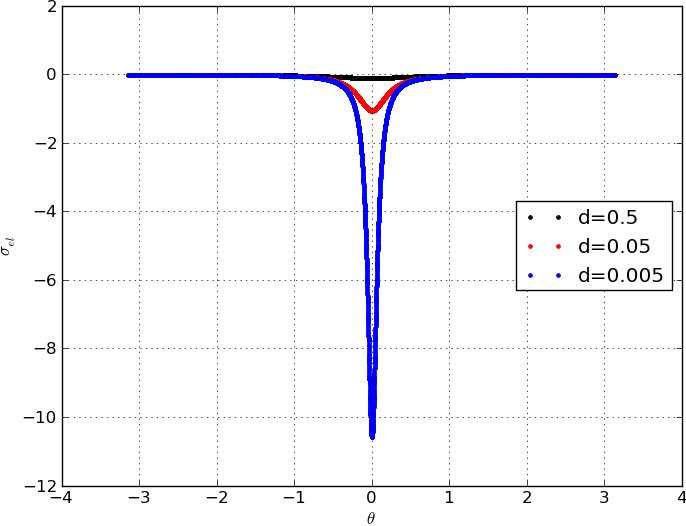}
 \end{center}
 \caption[Solution of integral equation for the elastance problem]{Solution of integral equation for the elastance problem, 
 $\sigma_{1,el} + \mu_{1,el}$ as a function of $d$.
\label{Fig:sigma1el}}
\end{figure}
\begin{figure}[H]
\begin{center}
 \includegraphics[width=6cm]{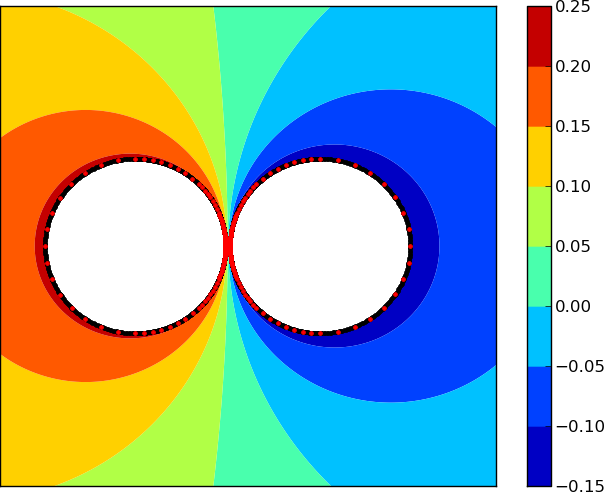}
 \end{center}
 \caption[Contour plot of $u$ in the exterior of the two discs]
 {Contour plot of $u$ in the exterior of the two discs for $\phi_{1} = 0.209$ and $\phi_{2} = -0.123$ for
 $d=0.05$.
\label{Fig:contourplot}}
\end{figure}
As noted earlier, the 
two disc Dirichlet problem has an analytic solution.
For this, suppose that the left disc is centered
at $\bx^{c}_{1} = \left(-1 -\frac{d}{2},0\right)$, that the right disc at $\bx^{c}_{2} = \left(1 + \frac{d}{2},0\right)$, and that the discs are held at 
constant potentials $\phi_{1}$ and $\phi_{2}$. Then, the exterior potential is given by
\begin{equation}
 u_{ex}\left(\bx\right) = -\frac{v_{1}}{2\pi}  
 \log \left(\frac{\left|\bx - \left(\alpha,0\right) \right|}{\left|\bx + \left(\alpha,0\right) \right|}\right) + v_{2}
\end{equation}
where 
\begin{equation}
\alpha = \sqrt{d + \frac{d^2}{4}},\quad v_{1} = \pi \frac{\left(\phi_{2} - \phi_{1} \right)}
{\log \left(\frac{\left|\bx_{0} + \left(\alpha,0\right) \right|}{\left|\bx_{0} - \left(\alpha,0\right) \right|}\right)},
\quad v_{2} = 0.5\left(\phi_{1} + \phi_{2}\right)
\end{equation}
with $\bx_{0} = \left(\frac{d}{2}\right)$. 
For each value of $d$, we compute the charge $q_{1}$ (since $q_{2} = -q_{1}$), 
the iteration count for the 
elastance problem $n_{it,el}$, and the relative $\mathbb{L}^{2}$ error of the potential on 
boundary $\Gamma_{i}$ given by $e_{i} = \sqrt{\frac{\int_{\Gamma_{i}}\left|u-u_{ex} \right|^2 \, ds_{\bx}}
{\int_{\Gamma_{i}}\left|u_{ex} \right|^2 \, ds_{\bx}}}$. We emphasize again that
this is {\em not} just a test of backward stability for the elastance solver, 
since we are solving two different boundary value problems.
\begin{table}[H]
\begin{center}
 \begin{tabular}{|c|c|c|c|c|}
 \hline
 $d$ & $q_{1}$ & $n_{it,el}$ & $e_{1}$ & $e_{2}$ \\ \hline
 0.5 & -0.239487 & 4 & $5.9 \, 10^{-8}$ & $1.5 \, 10^{-7}$ \\ \hline
 0.05 & -0.743917 & 8 & $2.0 \, 10^{-5}$ & $3.3 \, 10^{-5}$ \\ \hline
 0.005 & -2.348079 & 15 & $3.3 \, 10^{-5}$ & $5.1 \, 10^{-5}$ \\ \hline
 \end{tabular}
 \end{center}
 \caption[Summary of results for the capacitance and elastance problems 
with two discs]
 {Summary of results for the capacitance and elastance problems with two discs.}
 \label{capeltable}
\end{table}

\subsubsection{Splash test \label{sec:6.1.2}}
We repeat the test above with a more complicated geometry. 
We now consider $5$ conductors $D_{j}$, whose
boundaries $\Gamma_{j}$ are parametrized by
\begin{align}
 x_{j}\left(\theta \right) &=x^{c}_{j} +  r_{j}(\theta) \cos(\theta + \beta_{j}) \\
y_{j}\left(\theta \right) &=y^{c}_{j} +  r_{j}(\theta) \sin(\theta + \beta_{j})
\end{align}
where
\begin{align}
 r_{j}(\theta) = 1 + \sum_{k=1}^{12} a_{j,k} \sin \left(k\theta \right),
\end{align}
with the coefficients $a_{j,k}$ are uniformly chosen from $\left[0,0.1\right]$ and prescribe an arbitrary potential
on each of these objects.

We list here the parameters for defining the geometry and the exact solution 
in the previous section.
The table of centers $x^{c}_{j}, y^{c}_{j}$, and $\beta_{j}$ is given below.
\begin{table}[H]
\begin{center}
 \begin{tabular}{|c|c|c|c|c|c|}
  \hline
  & $\Gamma_{1}$ & $\Gamma_{2}$ & $\Gamma_{3}$ & $\Gamma_{4}$ & $\Gamma_{5}$ \\ \hline
 $x^{c}_{j}$ & -1.2 & 1.2 & 0 & -1.2 & 1.2  \\ \hline
 $y^{c}_{j}$ & 0 & 0 & -2.2 & -4.4 & -4.4  \\ \hline
 $\beta_{j}$ & $\pi$ & 0 & $\frac{\pi}{8}$ & $\frac{3\pi}{4}$ & -$\frac{\pi}{4}$ \\ \hline
 \end{tabular}
 \end{center}
 \caption[Geometry setup, splash test]{Parameters for setting up splash test for the Elastance and Mobility problems.}
 \end{table}
 In the next table, we list the coefficients $a_{j,k}$ for $j=1,2\ldots 5$ and $k = 1,2,\ldots 12$.
 \begin{table}[H]
\begin{center}
 \begin{tabular}{|c|c|c|c|c|}
  \hline
 $\Gamma_{1}$ & $\Gamma_{2}$ & $\Gamma_{3}$ & $\Gamma_{4}$ & $\Gamma_{5}$ \\ \hline
 0.012065 & 0.017038 & 0.070082 & 0.029959  & 0.012613 \\
 0.064385 & 0.041668 & 0.094629 & 0.069290 & 0.004017 \\
 0.006234 & 0.011991 & 0.046520 & 0.005102 & 0.07413 \\
 0.049028 & 0.022743 & 0.038905 & 0.067634 & 0.052361 \\
 0.030608 & 0.035266 & 0.043884 & 0.089215 & 0.084973 \\
 0.081641 & 0.10864 & 0.030143 & 0.097489 & 0.002916 \\
 0.099718 & 0.087338 & 0.084480 & 0.004693 & 0.081962 \\
 0.042460 & 0.096291 & 0.008018 & 0.055024 & 0.020443 \\
 0.076748 & 0.053323 & 0.069852 & 0.085238 & 0.069016 \\
 0.084684 & 0.040564 & 0.047617 & 0.070539 & 0.056950 \\
 0.016811 & 0.085034 & 0.015078 & 0.069771 & 0.051020 \\
 0.040454 & 0.016044 & 0.050553 & 0.051137 & 0.092286\\ \hline
 \end{tabular}
 \end{center}
 \caption[Geometry setup, splash test - II]{Coefficients $a_{j,k}$. For fixed $j$, the coefficients $a_{j,k}$ 
 for $\Gamma_{j}$ are listed in order of increasing $k$.}
 \end{table}
 
 The prescribed potentials $\phi_{j}$ on $\Gamma_{j}$ are given below,
as well as the $\mathbb{L}^{2}$ norms
 for the errors in $u$, 
given by $e_{i} = \sqrt{\frac{\int_{\Gamma_{i}}\left|u-u_{ex} \right|^2 \, ds_{\bx} }
{\int_{\Gamma_{i}}\left|u_{ex} \right|^2 \, ds_{\bx}}}$ on the boundary $\Gamma_{i}$.
The potential $u$ is computed after solving the 
elastance problem to see if we recover the exact values
$u_{ex}|_{\Gamma_{i}} = \phi_{i}$. 
\begin{table}[H]
\begin{center}
 \begin{tabular}{|c|c|c|c|c|c|}
  \hline
 j & 1 & 2 & 3 & 4 & 5 \\ \hline
 $\phi_{j}$ & 0.120625 & 0.643859 & 0.062342 & 0.490279 & 0.306079 \\ \hline
 $e_{j}$ & $2.1 \, 10^{-5}$ & $4.2 \, 10^{-6}$ & $2.4 \, 10^{-5}$ & 
$8.2 \, 10^{-6}$ & $8.0 \, 10^{-6}$ \\ \hline
 \end{tabular}
 \end{center}
 \caption[Results elastance splash test - III]{Prescribed potential on the boundary and the relative 
 $\mathbb{L}^{2}$ error in potential on the boundary.}
 \end{table}
 The elastance problem converged in 30 GMRES iterations, 
 with a relative residual of 
 $10^{-6}$. 
In Fig. \ref{Fig:contourplotsplash}, we show a 
contour plot of $u$, with boundary values set to the $\phi_{j}$.

\begin{figure}[H]
\begin{center}
 \includegraphics[width=6cm]{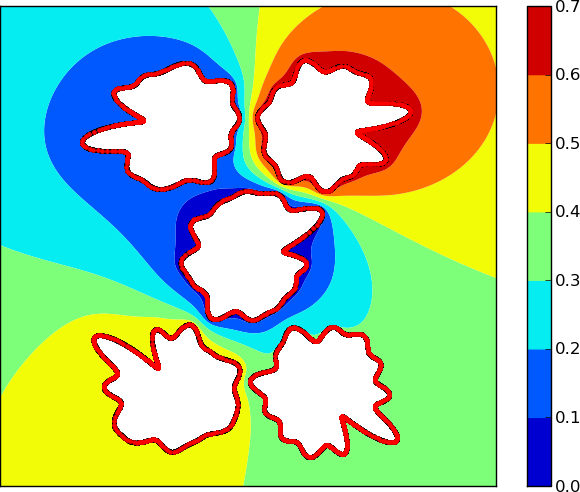}
 \end{center}
 \caption[Contour plot of the potential $u$ for the splash test for the elastance problem]
 {Contour plot of the potential $u$ 
 in the exterior of $\cup_{j} D_{j}$.
\label{Fig:contourplotsplash}}
\end{figure}

\subsubsection{Application: Computing dielectric properties of nanocomposites \label{sec:6.1.3}}
Nanocomposites are composite material consisting of nanoparticles in a host medium. Of particular interest are
nanocomposites consisting of metallic particles in a homogeneous organic host 
due to their applications in transformation optics and
high energy density storage materials. We shall treat the nanocomposite as
a collection of nanoparticles which are perfect conductors in ambient space. Computing bulk dielectric properties
of such materials as a function of shape, orientation and the volume fraction 
of these nanoparticles is of practical interest.
Low frequency dielectric constants are typically determined experimentally using ``capacitance'' measurements.
The dielectric constant is determined by measuring the voltage drop between two charged plates in the presence 
and absence of the nanocomposite. If the two conducting plates have charge $\pm Q$ and the 
measured potential difference is
$\Delta V$, then the ``capacitance'' of the configuration is computed as 
\begin{equation}
\tilde{C} = \frac{Q}{\Delta V} \label{eq:compeffcap}
\end{equation} 
The potential drop,
$\Delta V$ can be computed by solving an elastance problem.
\begin{rem}
It should be noted that obtaingin $\tilde{C}$ in this manner is different from computing
the mutual capacitance between the two plates for the given configuration
of nanoparticles \cite{zheng2012effects}.
Experimentally one could have applied a potential difference between the two plates 
and measured the charge accumulated on them. 
However, to determine the mutual capacitance of this configuration numerically, 
one would need to know the potentials on each of the nanoparticles, and this data is
not available.
\end{rem}

For fixed volume fraction, we carry out a two-dimensional version of the 
study in \cite{zheng2012effects}. 
In particular, we study the effects of varying the number of particles and
their aspect ratio.
Let $D_1$ and $D_2$ with boundaries $\Gamma_1$ and $\Gamma_2$,
represent the capacitor plates. The boundaries $\Gamma_1$ and $\Gamma_2$ are shifted copies
of a rounded bar $\gamma$ parametrized by
\begin{align}
 x\left(s \right) &=  \begin{cases}
  1.1\left(1 - \frac{2}{\pi}\left(e^{-100s^2} + s\cdot \text{erf}\left(10s\right)  \right)\right) \quad & s\in 
  \left[-\frac{\pi}{2},\frac{\pi}{2}\right]\\
  -x\left(2\pi-s \right) \quad & s\in (\frac{\pi}{2}, \frac{3\pi}{2}]
\end{cases} \\
y(s) &= \begin{cases}
         0.1 \text{erf}\left( 7 s\right) \quad & s\in\left[-\frac{\pi}{2},\frac{\pi}{2}\right]\\
         y\left(2\pi - s \right) \quad & s\in (\frac{\pi}{2}, \frac{3\pi}{2}]
        \end{cases}
\end{align}
The curve $\gamma$ is discretized by sampling it at
$s_k = -\frac{\pi}{2} + \pi(k-0.5)/N$, $k=1,2,\ldots 2N$.
We verify that the curve is well-resolved by studying the discrete 
Fourier coefficients of the sampled curve and choose $N$ sufficiently large that 
the curve is approximated to at least the desired tolerance.
More precisely,
the boundaries $\Gamma_1$ and $\Gamma_2$ are parametrized by 
$\left(x\left(s\right),y\left(s\right)\pm 1.1\right)$ 
and discretized using $800$ points each. 
For the nanoparticles, we use an 
$m\times 10$ lattice of elliptic
inclusions, each of which has an area equal to $\frac{0.002}{m}$ and an aspect ratio $A$. 
They are centered at
{\small 
\begin{equation}
\left(-0.9 + \frac{1.8\left(k-1\right)}{10}, -0.9 + \frac{1.8\left(j-1\right)}{m} \right) \quad j=1,2\ldots m, \, k=1,2\ldots 10 \, .
\end{equation}
}
 The aspect ratio $A$ 
is restricted to ensure that the nanoparticles do not overlap 
Each of these elliptical inclusions is discretized
using an equispaced sampling of the central angle with $600$ points each. 
\begin{figure}[H]
\centering
\includegraphics[scale=0.3]{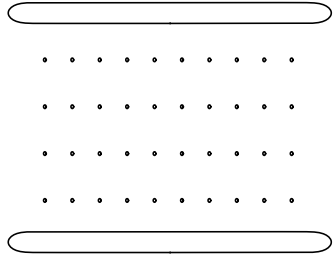}
\caption{Capacitor plates with intervening nanoparticles. Here, there are 
$m=4$ rows with aspect ratio set to $A=0.5$.
\label{Fig:3}}
\end{figure}
In this case, we have a total of $10m + 2$ conductors whose boundaries are discretzed
with $N_{pts} = 6000m + 1600$ points. We prescribe charges $1$ and $-1$ on conductors $\Gamma_1$ and $\Gamma_2$,
respectively, and assume the other $10m$ conductors are charge neutral. 
We measure the potential difference $\Delta V = V_1 - V_2$
where $V_i$ is the potential on $\Gamma_i$, $i=1,2$ and compute the capacitance via equation
\eqref{eq:compeffcap} for various values of $m$, $n$ and $A$. As before, we compute the layer potentials using a 6th
order GLQBX scheme accelerated via an FMM, and 
iterate using GMRES until the relative residual
in our computation is less than $10^{-6}$. 
\begin{table}[!ht]
\centering
 \begin{tabular}{|c|c|c|c|c|c|}
  \hline
  m & A & $N_{pts}$ & $N_{it}$ & $t_{solve}$ &  $\tilde{C}$ \\ \hline
  0 & - & 1600 & 8 &  0.4539 & 2.2949 \\ \hline
  \multirow{5}{*} 1 & 0.25 & \multirow{5}{*} {7600} & 11 & 3.2965 & 2.3147  \\ \hhline{~-~---}
 & 0.5 &  & 8 & 2.2417 & 2.3073 \\ \hhline{~-~---}
 & 1.0 &  & 8 & 2.2907 & 2.3033 \\ \hhline{~-~---}
 & 2.0 &  & 8 & 2.2167 & 2.3013 \\ \hhline{~-~---}
 & 4.0 &  & 11 & 3.4305 & 2.3003 \\ \hline
  \multirow{5}{*} 4 & 0.25 & \multirow{5}{*} {25600} & 11 & 11.5612 & 2.3191 \\ \hhline{~-~---}
 & 0.5 &  & 9 & 8.7257 & 2.3095 \\ \hhline{~-~---}
 & 1.0 &  & 7 & 7.2139 & 2.3047 \\ \hhline{~-~---}
 & 2.0 &  & 9 & 8.6497 & 2.3023 \\ \hhline{~-~---}
 & 4.0 &  & 10 & 10.1145 & 2.3012 \\ \hline
  \multirow{5}{*}{16} & 0.25 &  \multirow{5}{*} {97600} & 11 & 44.8332 & 2.3200\\ \hhline{~-~---}
 & 0.5 &  & 9 & 33.6639 & 2.3099 \\ \hhline{~-~---}
 & 1.0 &  & 8 & 31.2122 & 2.3049 \\ \hhline{~-~---}
 & 2.0 &  & 9 & 33.5139 & 2.3024 \\ \hhline{~-~---}
 & 4.0 &  & 11 & 46.4239 & 2.3012 \\ \hline

 \end{tabular}
\caption{Capacitance of nanocomposites. $N_{it}$ is the number of GMRES iterations, 
and $t_{solve}$ is the time taken
to solve the elastance problem in seconds on a single CPU core. }
\end{table}\\
Let $\tilde{C}_0$ be the capacitance in the absense of the nanocomposite (corresponding to $m=0$). We plot below
the percentage change in capacitance $100\left(\tilde{C} - \tilde{C}_{0}\right)/\tilde{C}_0$ as a function
of $m$ and $A$.

\begin{figure}[H]
\centering
\includegraphics[scale=0.4]{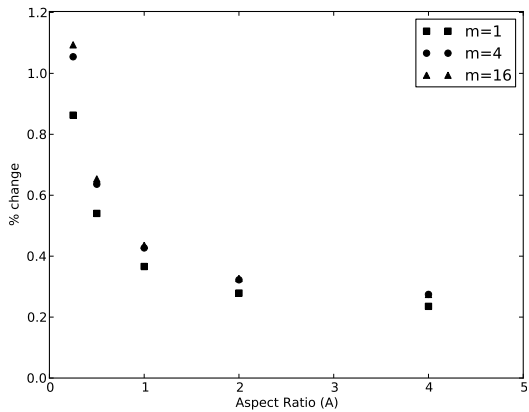}
\caption{ $\%$ change in $\tilde{C}$ as a function of $m$ and $A$.
\label{Fig:4}}
\end{figure}

\subsection{Mobility problem \label{sec:6.2}}
We turn now to a test for our mobility representation.
Given prescribed velocities for a set of rigid bodies, we solve 
the resistance problem
and compute the resulting forces and torques on them. We then 
use these forces and torques
as input for the mobility problem and check that the velocity on the boundary 
of the rigid body is the prescribed rigid body motion. As before, this is 
a stringent test, since the integral
equation for the resistance problem is {\em not} simply the inverse of the integral
equation for the mobility problem.

We consider, as above, the domain exterior to $N$ rigid bodies
$D_{i}$, whose boundaries are given by $\Gamma_{i}$. We prescribe
velocities $\mathbf{u} = \mathbf{v}_{i} + \omega_{i} \left(\bx - \bx^{c}_{i} \right)^{\perp}$
on $\Gamma_{i}$ and, solve the resistance problem
to compute the forces and torques 
on the rigid bodies $D_{i}$ and also the velocity at $\infty$, $\mathbf{u}_{\infty} = 
\lim_{\left|\bx \right| \to \infty} \mathbf{u} \left(\bx \right)$. 
We use these forces and torques as input for the mobility problem to compute the rigid
body motions. Let $\brho_{i,mob}$
denote the incident velocity field due to the forces and torques as described in section \ref{sec:intmob}
and let $\bmu_{i,mob}$ represent the unknown density on $\Gamma_{i}$ for the mobility problem. 
To summarize, we set
\begin{align}
\mathbf{u}\left(\mathbf{x}\right) & =\mathbf{u}_{inc}\left(\mathbf{x}\right)+\mathbf{u}_{sc}\left(\mathbf{x}\right)
+ \mathbf{u}_{\infty}
=\mathcal{S}_{\Gamma}\left({\bmu}_{mob} \left(\mathbf{x}\right)+{\brho}_{mob}\left(\mathbf{x}\right)\right) +
\mathbf{u}_{\infty}\, ,
\label{eq:TotRepresentationStokes2}
\end{align}
and wish to solve
\begin{align}
\left(\frac{1}{2}\mathbf{I}+\mathcal{K}+\mathbf{L}\right)\bmu_{mob} & =
-\left(\frac{1}{2}\mathbf{I}+\mathcal{K}\right)\brho_{mob} \, ,
\label{eq:ModRep2Stokes2}
\end{align}
where $\bmu_{mob} = \left(\bmu_{1,mob}, \bmu_{2,mob} \right)$,
$\brho_{mob} = \left(\brho_{1,mob}, \brho_{2,mob} \right)$, $\mathbf{u}_{\infty}$ is the velocity
at $\infty$ computed in the resistance problem, and the operators $\mathcal{K}$ and $\mathbf{L}$ are described
in section \ref{sec:extmob}. This is a small modification
of the representation presented in section \ref{sec:intmob} to account for the velocity
at $\infty$. After solving for $\bmu_{mob}$, we verify that $\bu$ in equation \eqref{eq:TotRepresentationStokes2}
is the original rigid body motion.
\subsubsection{Two discs}
We first test the mobility representation in the exterior of two discs,
with the same geometry is above, in section
\ref{sec:6.1.1}. However, we use a finer discretization to resolve the more singular
densities incurred in the mobility problem.
\begin{figure}[H]
\begin{center}
 \includegraphics[width=6cm]{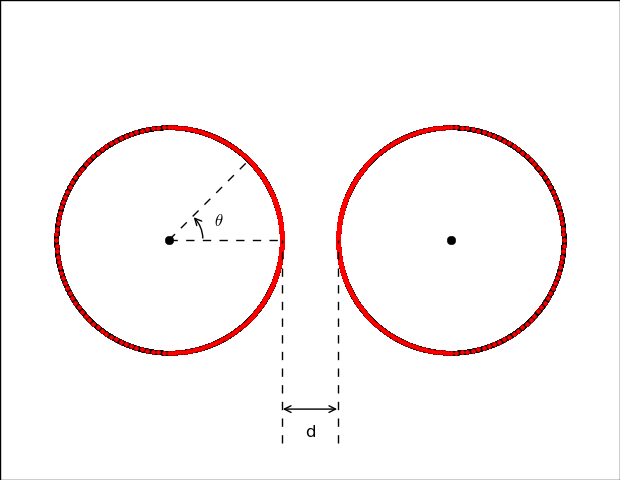}
 \end{center}
 \caption[Two discs problem - Mobility]{Discretization of the discs for Mobility example.}\label{Fig:twocd_mob}
\end{figure}
Refering to Fig. \ref{Fig:twocd_mob}, we set $\mathbf{u}|_{\Gamma_{i}} = 
\mathbf{v}_{i} + \omega_{i} \left(\bx - \bx^{c}_{i} \right)^{\perp}$ for $i=1,2$ ($i=$ corresponds to the disc
on the left), where we set
$\mathbf{v}_{1} = \left(2.09, 1.00 \right)$, $\mathbf{v}_{2} = \left(-1.034,0.254 \right)$,  $\omega_{1} = 0.12$
and $\omega_{2} = 0.33$. We again test the problem for $d=0.5,0.05,0.005$. 

As above, we use
a Nystr\"{o}m discretization with Gauss-Legendre panels.
$\mathbf{s}_{i,j,l}$ denotes the $j$th Gauss-Legendre node
on panel $i$ on boundary $l$. Let
$\brho_{i,j,l,mob},\bmu_{i,j,l,mob}$ denote 
the densities at $\mathbf{s}_{i,j,l}$.
We use the GLQBX quadrature scheme for evaluating the layer potential,
$\frac{1}{2}\mathbf{I} + \mathcal{K}$ in equation \eqref{eq:ModRep2Stokes2}. 
We use an iterative GMRES-based solver
to obtain $\bmu_{mob}$, with a relative
residual tolerance of $10^{-6}$. The physical conditioning increases as 
$d\to 0$, requiring a large number of iterations. 
To improve the rate of convergence of GMRES, we use $L^2$ weighting for
the unknowns \cite{bremer2010efficient}.
That is, we use $\bmu^{scale}_{mob} = \bmu_{i,j,l,mob} \sqrt{r_{i}}$ as unknowns. 

We solve the following linear system
\begin{equation}
D \left(\frac{1}{2}\mathbf{I}+\mathcal{\tilde{K}}+\mathbf{\tilde{L}}\right) D^{-1} \bmu_{mob}^{scale} =
-D \left(\frac{1}{2}\mathbf{I}+\mathcal{\tilde{K}}\right)\brho_{mob} \, ,
\end{equation}
where $\mathcal{\tilde{K}}$ and $\mathbf{\tilde{L}}$ are discretized versions of $\mathcal{K}$ and $\mathbf{L}$ 
respectively, and $D$ is the diagonal operator given by $D \bmu_{mob} = \bmu^{scale}_{mob}$. 

We plot below the 
net surface traction $\brho_{1,mob} +
\bmu_{1,mob}$ and a quiver plot of the velocity field in the exterior of the two
discs for $d=0.05$ (Figs.  \ref{brhodensmob} and \ref{Fig:contourplotmob}).
From symmetry considerations, 
$\brho_{2,mob} + \bmu_{2,mob} = -\left(\brho_{1,mob} + \bmu_{1,mob} \right)$.
\begin{figure}[H]
\begin{center}
\includegraphics[width=5.5cm]
{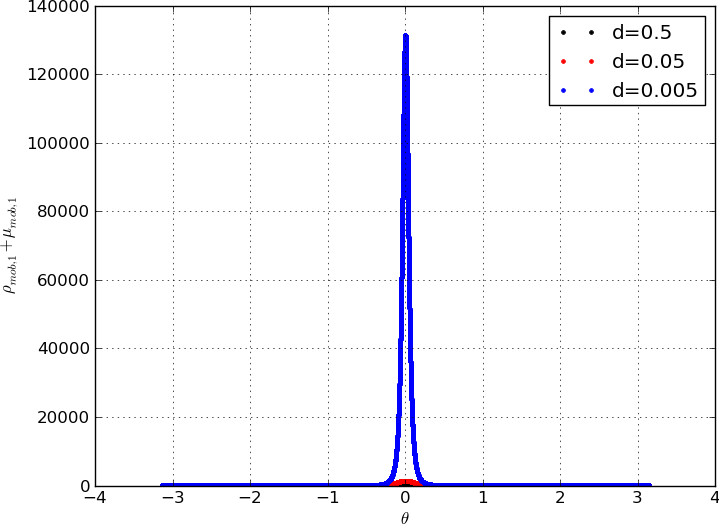}
\hspace{.3in}
    \includegraphics[width=5.5cm]
{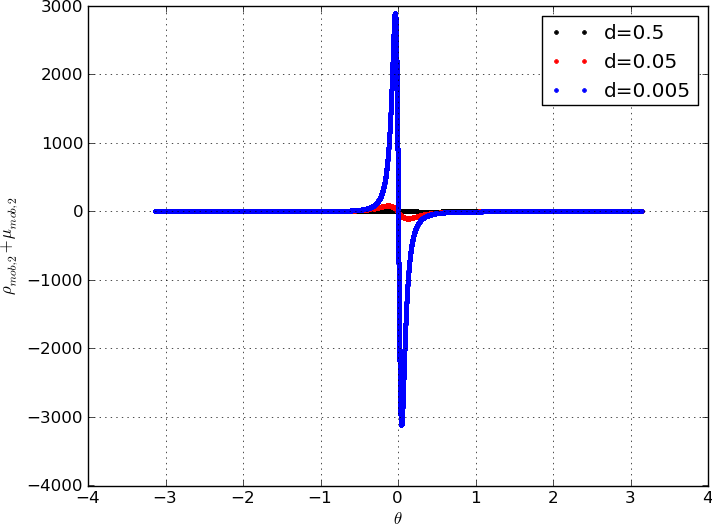}
\end{center}
\caption[Solution for integral equation for Mobility problem]
{Integral equation solution for the mobility problem:
$\rho_{1,1,mob}+\mu_{1,1,mob}$ (left) and $\rho_{2,1,mob}+\mu_{2,1,mob}$ (right) as a function of $\theta$ 
for $d=0.5,0.05,0.005$}
\label{brhodensmob}
\end{figure}
\begin{figure}[H]
\begin{center}
 \includegraphics[width=6cm]{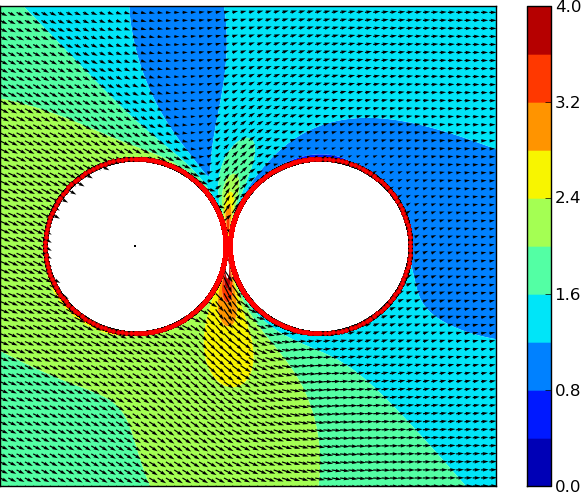}
 \end{center}
 \caption[Quiver plot of $\mathbf{u}$ and contour plot of $\left| \mathbf{u} \right|$ in the exterior of the two discs]
 {Quiver plot of $\mathbf{u}$ and contour plot of $\left|\mathbf{u} \right|$ in the exterior of the two discs.
\label{Fig:contourplotmob}}
\end{figure}
For each value of $d$, we compute the forces and torques $F_{1,1}, F_{2,1}, T_{1}, T_{2}$ 
(since $\mathbf{F}_{2} = -\mathbf{F}_{1}$), the iteration count for the mobility problem, $n_{it}$, 
and the relative $\mathbb{L}^{2}$ error of the velocity on both
the boundaries $\Gamma_{1}$ and $\Gamma_{2}$ given by $e_{i} = \sqrt{\frac{\int_{\Gamma_{i}}\left|\mathbf{u}-\mathbf{u}_{ex} \right|^2 \, ds_{\bx}}
{\int_{\Gamma_{i}}\left|\mathbf{u}_{ex} \right|^2 \, ds_{\bx}}}$. This is again a nontrivial test
of our solvers, as we enforce
boundary conditions on the fluid stress here, not the velocity $\mathbf{u}$.

\begin{table}[H]
\begin{center}
{\footnotesize
 \begin{tabular}{|c|c|c|c|c|c|c|c|}
 \hline
 $d$ & $F_{1,1}$ & $F_{2,1}$ & $T_{1}$ & $T_{2}$ & $n_{it}$ & $e_{1}$ & $e_{2}$  \\ \hline
 0.5 & 27.180434 & -6.575686 & -1.496082 & 1.494675 & 7 & $8.8 \, 10^{-8}$ & $1.8 \, 10^{-5}$ \\ \hline
 0.05 & 499.08688 & -15.202716 & -11.159661 & -4.859692 & 19 & $5.1 \, 10^{-6}$ & $8.5 \, 10^{-6}$ \\ \hline
 0.005 & 14653.544 & -40.877338 & -42.867299 & -24.078713 & 60 & $1.0 \, 10^{-6}$ & $2.2 \, 10^{-6}$ \\ \hline
 \end{tabular}
 \caption[Summary of results for two discs test for resistance and mobility problems]
 {Summary of results for two discs test for resistance and mobility problems.}
 \label{resmobtable}
}
 \end{center}
\end{table}

\subsubsection{Splash test \label{sec:6.1.2a}}
We repeat the test above in a more complicated geometry, but 
consider $5$ bodies $D_{j}$, whose
boundaries $\Gamma_{j}$ are parametrized as
\begin{align}
 x_{j}\left(\theta \right) &=x^{c}_{j} +  r_{j}(\theta) \cos(\theta + \beta_{j}) \\
y_{j}\left(\theta \right) &=y^{c}_{j} +  r_{j}(\theta) \sin(\theta + \beta_{j})
\end{align}
where
\begin{align}
 r_{j}(\theta) = 1 + \sum_{k=1}^{12} a_{j,k} \sin \left(k\theta \right)
\end{align}
where the parameters $a_{j,k}$, $x^{c}_{j}$ and $y^{c}_{j}$ are the described in section
\ref{sec:6.1.2}.

 The prescribed velocities $\mathbf{v}_{j}=\left(v_{1,j}, v_{2,j}\right)$ and $\omega_{j}$ 
 on $\Gamma_{j}$ are given below, along with the $\mathbb{L}^{2}$ norm
 in the error in $\mathbf{u}$, given by $e_{i} = \sqrt{\frac{\int_{\Gamma_{i}}\left|\mathbf{u}-\mathbf{u}_{ex} \right|^2 \, ds_{\bx}}
{\int_{\Gamma_{i}}\left|\mathbf{u}_{ex} \right|^2 \, ds_{\bx}}}$ on the boundary $\Gamma_{i}$, 
after solving the mobility problem is listed below where $\mathbf{u}_{ex}|_{\Gamma_{i}} = \mathbf{v}_{i} + \omega_{i}
\left(\bx - \bx^{c}_{i} \right)^{\perp}$. 
\begin{table}[H]
\begin{center}
 \begin{tabular}{|c|c|c|c|c|c|}
  \hline
 j & 1 & 2 & 3 & 4 & 5 \\ \hline
 $v_{1,j}$ & -0.379375 & -0.009720 & 0.497180 & 0.346837 & -0.197527 \\ \hline
 $v_{2,j}$ & 0.143846 & -0.193921 & -0.075401 & -0.331891 & 0.273004 \\ \hline
 $\omega_{j}$ & -0.437658 & 0.316414 & 0.267477 & -0.095456 & -0.184353 \\ \hline
 $e_{j}$ & $1.8 \, 10^{-5}$ & $2.5 \, 10^{-5}$ & $1.1 \, 10^{-5}$ & $1.3 \, 10^{-5}$ & 
$1.3 \, 10^{-5}$ \\ \hline
 \end{tabular}
 \end{center}
 \caption[Results elastance splash test - III]{Prescribed velocity on the boundary and the relative 
 $\mathbb{L}^{2}$ error in velocity on the boundary after solving the resistane 
 and mobility problems.}
 \end{table}
 The mobilitiy problem converged in 71 GMRES iterations 
 to a relative tolerance of $10^{-6}$.
 We show a quiver plot for the velocity $\mathbf{u}$ corresponding to 
 the above prescribed value of velocity on the boundary $\mathbf{u}|_{\Gamma_{j}}$.
 The background
 is a contour plot of the magnitude $\left|\mathbf{u} \right|$ in the exterior
 of the $\cup_{j} D_{j}$ (Fig.  \ref{Fig:contourplotmobsplash}).
\begin{figure}[H]
\begin{center}
 \includegraphics[width=6cm]{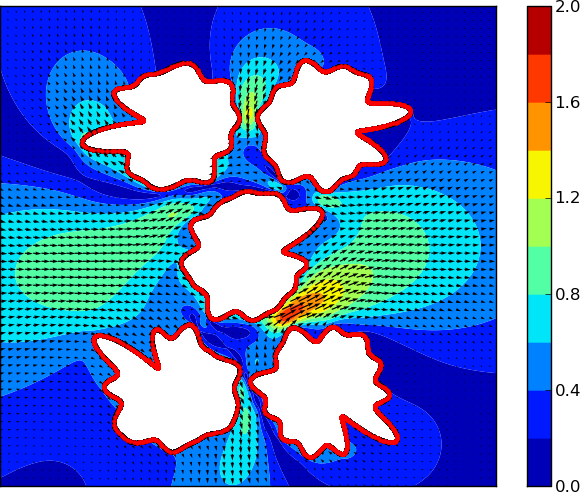}
 \end{center}
 \caption[Quiver plot of $\mathbf{u}$ superimposed on contour plot of $\left| \mathbf{u} \right|$ for splash test for mobility problem]
 {Quiver plot of $\mathbf{u}$ superimposed on contour plot of $\left|\mathbf{u} \right|$ in the exterior of $\cup_{j} D_{j}$.
\label{Fig:contourplotmobsplash}}
\end{figure} 

\section{Conclusions} \label{sec:concl}

We have derived a new, physically motivated integral formulation for 
the elastance problem in exterior domains. The analogous physical reasoning
yields a new derivation of an integral equation developed earlier by Kim and Karrila
for the mobility problem \cite{karrila1989integral}.
Discretization of the resulting integral equations using 
the quadrature scheme GLQBX \cite{rachh15,GLQBX} permits high order
accuracy to be obtained in complex geometry, including the interaction
of close-to-touching boundary components.
The resulting linear systems can be solved iteratively using GMRES and
the necessary matrix-vector multiplications can be accelerated using the 
fast multipole method. If $N$ denotes the number of points used in the 
discretization of the physical boundaries, the total cost scales linearly
with $N$. We are currently working on the extension of our scheme to 
closed or periodic systems and to problems in three dimensions. 

\section*{Acknowledgments}
This work was supported in part by the Applied Mathematics program
in the Department of Energy, Office of Advanced Scientific Computing Research, 
under contract DEFGO288ER25053 and
by the Office of the Assistant Secretary of Defense for Research and 
Engineering and AFOSR under NSSEFF Program Award FA9550-10-1-0180. The authors would
like to thank Alex Barnett and Mike O'Neil for several useful discussions.

\bibliography{Elastance8}
\bibliographystyle{ieeetr}
\end{document}